\documentclass[12pt]{amsart}

\usepackage{amsmath}
\usepackage{amsfonts}
\usepackage{amssymb}
\usepackage{latexsym}
\usepackage[all]{xy}
\newtheorem{theorem}{Theorem}[section]
\newtheorem{proposition}[theorem]{Proposition}
\newtheorem{corollary}[theorem]{Corollary}

\newtheorem{remark}[theorem]{Remark}
\theoremstyle{definition}

\numberwithin{equation}{section}

\def\sD{{\mathfrak D}}      
   \def\sH{{\mathfrak H}}   
   \def\sK{{\mathfrak K}}   \def\sL{{\mathfrak L}}
\def\sM{{\mathfrak M}}   \def\sN{{\mathfrak N}}

      \def\dC{{\mathbb C}}
\def\dD{{\mathbb D}}

   \def\dN{{\mathbb N}}   
      
   \def\dT{{\mathbb T}}   
      
   \def\dZ{{\mathbb Z}}

\def\cA{{\mathcal A}}      
      
   \def\cH{{\mathcal H}}   
      \def\cL{{\mathcal L}}
\def\cM{{\mathcal M}}      
      
\def\cS{{\mathcal S}}      
\def\cV{{\mathcal V}}

\def\bL{{\mathbf L}}

\def\clos{{\rm clos\,}}

\def\wt{\widetilde}

\def\f{\varphi}
\def\bl{\bigl}

\def\uphar{{\upharpoonright\,}}

\def\ran{{\rm ran\,}}
\def\dom{{\rm dom\,}}
\def\cran{{\rm \overline{ran}\,}}
\def\cspan{{\rm \overline{span}\, }}
\begin{document}
\title
[Realizations of the Schur iterates]
{Iterates of the Schur class
operator-valued function and their conservative realizations}
\author{
Yury~Arlinski\u{i}}
\address{Department of Mathematical Analysis \\
East Ukrainian National University \\
Kvartal Molodyozhny 20-A \\
Lugansk 91034 \\
Ukraine} \email{yma@snu.edu.ua} 

\begin{abstract}

Let $\mathfrak M$ and $\mathfrak N$ be separable Hilbert spaces and
let $\Theta(\lambda)$ be a function from the Schur class ${\bf
S}(\mathfrak M,\mathfrak N)$ of contractive functions holomorphic on
the unit disk. The operator generalization of the classical Schur
algorithm associates with $\Theta$ the sequence of contractions (the
Schur parameters of $\Theta$) $
 \Gamma_0=\Theta(0)\in \bL(\sM,\sN),\;
\Gamma_n\in\bL(\sD_{\Gamma_{n-1}},\sD_{\Gamma^*_{n-1}})
$
and the sequence of functions $\Theta_0=\Theta$, $\Theta_n\in {\bf
S}(\sD_{\Gamma_n},\sD_{\Gamma^*_n})$ $ n=1,\ldots$ (the Schur iterares of $\Theta$) connected by the
relations
\[
\Gamma_n=\Theta_n(0),\;
\Theta_n(\lambda)=\Gamma_n+\lambda D_{\Gamma^*_n}\Theta_{n+1}(\lambda)(I+\lambda\Gamma^*_n\Theta_{n+1}(\lambda))^{-1}D_{\Gamma_n},\;|\lambda|<1.
\]
The function $\Theta(\lambda)\in {\bf S}(\sM,\sN)$ can be realized as the
transfer function
\[
\Theta(\lambda)=D+\lambda C(I-\lambda A)^{-1}B
\]
 of a linear
conservative and simple discrete-time system
$\tau=\left\{\begin{bmatrix}D&C\cr B& A\end{bmatrix};
 \mathfrak M,
\mathfrak N,\mathfrak H\right\}$
with the state space $\mathfrak H$ and the input and output spaces
$\mathfrak M$ and $\mathfrak N $, respectively.

 In this paper we give a construction of conservative and
 simple realizations of the Schur iterates $\Theta_n$ by means of the
 conservative and simple realization of $\Theta$.

\end{abstract}
\maketitle
 \tableofcontents
\section{Introduction}
The Schur class ${\bf S}$ of scalar analytic functions and bounded
by one in the unit disc $\dD=\{\lambda\in\dC:|\lambda|<1\}$ plays a
prominent role in complex analysis and operator theory as well in
their applications in linear system theory and mathematical
engineering. Given a Schur function $f(\lambda)$, which is not a
finite Blaschke product, define inductively
\[
f_0(\lambda)=f(\lambda),\;
f_{n+1}(\lambda)=\frac{f_n(\lambda)-f_n(0)}{\lambda(1-\overline{f_n(0)}f_n(\lambda))},\;
n\ge 0.
\]
It is clear that  $\{f_n\}$ is an {\it infinite} sequence of Schur
functions called the \textit{$n-th$ Schur iterates} and neither of
its terms is a finite Blaschke product. The numbers
$\gamma_n:=f_n(0)$ are called the {\it Schur parameters:}
\[
\cS f=\{\gamma_0,\gamma_1,\ldots\}.
\]
Note that
\[
f_n(\lambda)=\frac{\gamma_n+\lambda
f_{n+1}(\lambda)}{1+\bar\gamma_n\lambda
f_{n+1}}=\gamma_n+(1-|\gamma_n|^2)\frac{\lambda
f_{n+1}(\lambda)}{1+\bar\gamma_n\lambda f_{n+1}(\lambda)},\; n\ge 0.
\]
The method of labeling $f\in{\bf S}$ by its Schur parameters is
known as the \textit{Schur algorithm} and is due to I.~Schur
\cite{Schur}. In the case when
\[
f(\lambda)=e^{i\f}\prod_{k=1}^N \frac{\lambda-\lambda_k}{1-\bar
\lambda_k \lambda}
\]
is a finite Blaschke product of order $N$, the Schur algorithm
terminates at the $N$-th step. The sequence of Schur parameters
$\{\gamma_n\}_{n=0}^N$ is finite, $|\gamma_n|<1$ for
$n=0,1,\ldots,N-1$, and $|\gamma_N|=1$.

The Schur algorithm for matrix valued Schur class functions  has
been considered in the paper of Delsarte, Genin, and Kamp \cite{DGK}
and in the book of Dubovoj, Fritzsche, and Kirstein \cite{DFK}. An
operator extension of the Schur algorithm was developed by
T.~Constantinescu in \cite{Const} and with numerous applications is
presented in the book of Bakonyi and Constantinescu \cite{BC}.

In what follows the class of all continuous linear operators defined
on a complex Hilbert space $\sH_1$ and taking values in a complex
Hilbert space $\sH_2$ is denoted by $\bL(\sH_1,\sH_2)$ and
${\bL}(\sH):= {\bL}(\sH,\sH)$. The domain, the range, and the
null-space of a linear operator $T$ are denoted by $\dom T$, $\ran
T$, and $\ker T$, respectively. The set of all regular points of a
closed operator $T$ is denoted by $\rho(T)$. We denote by $I_\cH$
the identity operator in a Hilbert space $\cH$ and by $P_\cL$ the
orthogonal projection onto the subspace (the closed linear manifold)
$\cL$. The notation $T\uphar \cL$ means the restriction of a linear
operator $T$ on the set $\cL$. The positive integers will be denoted
by $\dN$. An operator $T\in\bL(\sH_1,\sH_2)$ is said to be
\begin{enumerate}\def\labelenumi{\rm (\alph{enumi})}
\item \textit{contractive} if $\|T\|\le 1$;

\item \textit{isometric} if $\|Tf\|=\|f\|$ for all $f\in \sH_1$
$\iff T^*T=I_{\sH_1}$;

\item \textit{co-isometric} if $T^*$ is isometric $\iff
TT^*=I_{\sH_2}$;
\item \textit{unitary} if it is both isometric and co-isometric.
\end{enumerate}
Given a contraction $T\in \bL(\sH_1,\sH_2)$. The operators
\[
D_T:=(I-T^*T)^{1/2},\qquad D_{T^*}:=(I-TT^*)^{1/2}
\]
are called the \textit{defect operators} of $T$, and the subspaces
$\sD_T=\cran D_T,$ $\sD_{T^*}=\cran D_{T^*}$ the \textit{defect
subspaces} of $T$. The dimensions $\dim\sD_T,$ $\dim\sD_{T^*}$ are
known as the \textit{defect numbers} of $T$. The defect operators
satisfy the following intertwining relations
\begin{equation}
\label{defect} TD_{T}=D_{T^*}T,\qquad T^*D_{T^*}=D_{T}T^*.
\end{equation}
It follows from \eqref{defect} that $T\sD_T\subset\sD_{T^*}$,
$T^*\sD_{T^*}\subset\sD_T$, and $T(\ker D_T)=\ker D_{T^*},$
$T^*(\ker D_{T^*})=\ker D_{T}$. Moreover, the operators $T\uphar\ker
D_{T}$ and $T^*\uphar\ker D_{T^*}$ are isometries and $T\uphar\sD_T$
and $T^*\uphar\sD_{T^*}$ are \textit{pure} contractions, i.e.,
$||Tf||<||f||$ for $f\in\sH\setminus\{0\}$.

The \textit{Schur class} ${\bf S}(\sH_1,\sH_2)$ is the set of all
function $\Theta(\lambda)$ analytic on the unit disk $\dD$ with
values in $\bL(\sH_1,\sH_2)$ and such that $\|\Theta(\lambda)\|\le
1$ for all $\lambda\in\dD$. The following theorem takes place.
\begin{theorem}
\label{MO}\cite{Const}, \cite{BC}. Let $\sM$ and $\sN$ be separable
Hilbert spaces and let the function $\Theta(\lambda)$ be from the
Schur class ${\bf S}(\sM,\sN).$ Then there exists a function
$Z(\lambda)$ from the Schur class ${\bf
S}(\sD_{\Theta(0)},\sD_{\Theta^*(0)})$ such that
\begin{equation}
\label{MREP}
\Theta(\lambda)=\Theta(0)+D_{\Theta^*(0)}Z(\lambda)(I+\Theta^*(0)Z(\lambda))^{-1}D_{\Theta(0)},\;\lambda\in\dD.
\end{equation}
\end{theorem}
In what follows we will call the representation \eqref{MREP} of a
function $\Theta(\lambda)$ from the Schur class \textit{the M\"obius
representation} of $\Theta(\lambda)$ and the function $Z(\lambda)$
we will call \textit{the M\"obius parameter} of $\Theta(\lambda)$.
Clearly, $Z(0)=0$ and by Schwartz's lemma we obtain that
\[
||Z(\lambda)||\le|\lambda|,\;\lambda\in\dD.
\]

\textit{The operator Schur's algorithm} \cite{BC}. Fix
$\Theta(\lambda)\in{\bf S}(\sM,\sN)$, put
$\Theta_0(\lambda)=\Theta(\lambda)$ and let $Z_0(\lambda)$ be the
M\"obius parameter of $\Theta$. Define
\[
\Gamma_0=\Theta(0),\;
\Theta_1(\lambda)=\frac{Z_0(\lambda)}{\lambda}\in {\bf
S}(\sD_{\Gamma_0},\sD_{\Gamma^*_0}),\;\Gamma_1= \Theta_1(0)=Z'_0(0).
\]
If $\Theta_0(\lambda),\ldots,\Theta_n(\lambda)$ and
$\Gamma_0,\ldots, \Gamma_n$ have been chosen, then let
$Z_{n+1}(\lambda)\in {\bf S}(\sD_{\Gamma_n},\sD_{\Gamma^*_n})$ be
the M\"obius parameter of $\Theta_n$. Put
\[
\Theta_{n+1}(\lambda)=\frac{Z_{n+1}(\lambda)}{\lambda},\;
\Gamma_{n+1}=\Theta_{n+1}(0).
\]
 The contractions $\Gamma_0\in\bL(\sM,\sN),$
$\Gamma_n\in\bL(\sD_{\Gamma_{n-1}},\sD_{\Gamma^*_{n-1}})$,
$n=1,2,\ldots$ are called the \textit{Schur parameters} of
$\Theta(\lambda)$ and the function $\Theta_n(\lambda) \in {\bf
S}(\sD_{\Gamma_{n-1}},\sD_{\Gamma^*_{n-1}})$ we will call the $n-th$
\textit{Schur iterate} of  $\Theta(\lambda)$.

Formally we have
\[
\Theta_{n+1}(\lambda)\uphar\ran
D_{\Gamma_n}=\frac{1}{\lambda}D_{\Gamma^*_n}(I_{\sD_{\Gamma^*_n}}-\Theta_n(\lambda)\Gamma^*_n)^{-1}
(\Theta_n(\lambda)-\Gamma_n)D^{-1}_{\Gamma_n}\uphar\ran
D_{\Gamma_n}.
\]
Clearly, the sequence of Schur parameters $\{\Gamma_n\}$ is infinite
if and only if the operators $\Gamma_n$ are non-unitary. The
sequence of Schur parameters consists of finite number operators
$\Gamma_0,$ $\Gamma_1,\ldots, \Gamma_N$ if and only if
$\Gamma_N\in\bL(\sD_{\Gamma_{N-1}},\sD_{\Gamma^*_{N-1}})$ is
unitary. If $\Gamma_N$ is isometric (co-isometric) then $\Gamma_n=0$
for all $n>N$.

The following theorem is the operator generalization of Schur's result.
\begin{theorem} \label{SchurAlg}\cite{Const}, \cite{BC}. There is a one-to-one
correspondence between the Schur class functions ${\bf S}(\sM,\sN)$
and the set of all sequences of contractions $\{\Gamma_n\}_{n\ge 0}$
such that
\begin{equation}
\label{CHSEQ} \Gamma_0\in\bL(\sM,\sN),\;\Gamma_n\in
\bL(\sD_{\Gamma_{n-1}},\sD_{\Gamma^*_{n-1}}),\; n\ge 1.
\end{equation}
\end{theorem}
Notice that a sequence of
contractions of the form \eqref{CHSEQ} is called the \textit{choice
sequence} \cite{CF}.

It is known \cite{Br1}, \cite{A} that every $\Theta(\lambda)\in {\bf
S}(\sM,\sN)$ can be realized as the transfer function
\[
\Theta(\lambda)=D+\lambda C(I_\sH-\lambda A)^{-1}B
\]
of a linear conservative and simple discrete-time system (see
Section \ref{secS})
\[
\tau=\left\{\begin{bmatrix}D&C\cr B&
A\end{bmatrix};\sM,\sN,\sH\right\}
\]
with the state space $\sH$ and input and output spaces $\sM$ and
$\sN$, respectively. In this paper we study the problem of the
conservative realizations of the Schur iterates of the function
$\Theta(\lambda)\in{\bf S}(\sM,\sN)$ by means of the the
conservative realization of $\Theta$.

In this connection it should be pointed out that the similar problem
for a scalar generalized Schur class function has been studied in
papers \cite{AADL1}, \cite{AADLW1}, \cite{AADLW2}, \cite{ADL}.

Here we describe our main results. Let $A$ be a completely
non-unitary contraction \cite{SF} in a separable Hilbert space
$\sH$. Define the subspaces and operators
\[
\begin{array}{l}
\sH_{m,0}=\ker D_{A^m},\;\sH_{0,l}=\ker D_{A^{*l}},\\
\sH_{m,l}=\ker D_{A^{m}}\cap\ker D_{A^{*l}},\; m,l\in\dN,\\
A_{m,l}=P_{m,l}A\uphar\sH_{m,l},
\end{array}
\]
where $P_{m,l}$ is the orthogonal projection in $\sH$ onto
$\sH_{m,l}$.

We prove that

1) if $A$ is a completely non-unitary contraction in a Hilbert space
then for every $n\in\dN$ the operators
\[
A_{n,0},\;A_{n-1,1},\ldots, A_{0,n}
\]
are unitary equivalent completely non-unitary contractions and their
Sz.-Nagy-- Foias characteristic functions \cite{SF} coincide with
the pure contractive part \cite{SF}, \cite{BC} for the $n$-th Schur
iterate $\Phi_n(\lambda)$ of the characteristic function
$\Phi(\lambda)$ of $A$;

2) if $\Theta(\lambda)\in {\bf S}(\sM,\sN)$ is the transfer function
of a simple conservative system
\[
\tau=\left\{\begin{bmatrix}\Gamma_0&C\cr B&
A\end{bmatrix};\sM,\sN,\sH\right\}
\]
then the Schur parameters of $\Theta$ take the form
\[
\begin{array}{l}
\Gamma_1=D^{-1}_{\Gamma^*_0}C\left(D^{-1}_{\Gamma_0}B^*\right)^*,�\;
\Gamma_2=D^{-1}_{\Gamma^*_1}D^{-1}_{\Gamma^*_0}CA
\left(D^{-1}_{\Gamma_1}D^{-1}_{\Gamma_0}\left(B^*\uphar\sH_{1,0}\right)\right)^*,\ldots,\\
\Gamma_n=D^{-1}_{\Gamma^*_{n-1}}\cdots
D^{-1}_{\Gamma^*_0}CA^{n-1}\left(D^{-1}_{\Gamma_{n-1}}\cdots
D^{-1}_{\Gamma_0}\left(B^*\uphar\sH_{n-1,0}\right)\right)^*,\ldots,
\end{array}
\]
and the $n$-th Schur iterate $\Theta_n(\lambda)$ of $\Theta$ is the
transfer function of the simple conservative and unitarily
equivalent systems
\[
\tau^{(k)}_{n}=\left\{\begin{bmatrix}\Gamma_n&D^{-1}_{\Gamma^*_{n-1}}
\cdots D^{-1}_{\Gamma^*_{0}}(CA^{n-k})\cr
A^k\left(D^{-1}_{\Gamma_{n-1}}\cdots D^{-1}_{\Gamma_{0}}
\left(B^*\uphar\sH_{n,0}\right)\right)^*&A_{n-k,k}\end{bmatrix};
\sD_{\Gamma_{n-1}},\sD_{\Gamma^*_{n-1}}, \sH_{n-k,k}\right\}
\]
 for $k=0,\ldots, n$.
Here $D^{-1}_{\Gamma_m}$ and $D^{-1}_{\Gamma^*_m}$ are the Moore--
Penrose pseudo-inverses.  For a completely non-unitary contraction
$A$ with rank one defect operators it was proved in \cite{AGT} that
the characteristic functions of the operators $A_{1,0}=P_{\ker
D_A}A\uphar\ker D_A$ and $A_{0,1}=P_{\ker D_{A^*}}A\uphar\ker
D_{A^*}$ coincide with the first Schur iterate of the characteristic
function of $A$. This result has been established using the model of
$A$ given by a truncated CMV matrix. Here we use another approach
based on the parametrization of a contractive block-operator matrix
\[
T=\begin{bmatrix} D&C \cr B&A\end{bmatrix} :
\begin{array}{l} \sM \\\oplus\\ \sH \end{array} \to
\begin{array}{l} \sN \\\oplus\\ \sK \end{array}
\]
established in  \cite{AG}, \cite{DaKaWe}, \cite {ShYa}, and the
construction of the passive realization for the M\"obius parameter
of $\Theta(\lambda)$ obtained in \cite{ARL1} by means of a passive
realization of $\Theta$.
\section{Completely non-unitary contractions}

Let $S$ be an isometry in a separable Hilbert space $H$. A subspace
$\Omega$ in $H$ is called wandering for $V$ if $S^p\Omega\perp
S^q\Omega$ for all $p,q\in \dZ_+$, $p\ne q$. Since $S$ is an
isometry, the latter is equivalent to $S^n\Omega\perp \Omega$ for
all $n\in\dN$. If $H=\sum_{n=0}^\infty\oplus S^n\Omega$ then $S$ is
called a \textit{unilateral shift} and $\Omega$ is called the
generating subspace. The dimension of $\Omega$ is called the
multiplicity of the unilateral shift $S$. It is well known
\cite[Theorem I.1.1]{SF} that $S$ is a unilateral shift if and only
if $\bigcap_{n=0}^\infty S^n H=\{0\}.$ Clearly, if an isometry $V$
is the unilateral shift in $H$ then $\Omega=H\ominus SH$ is the
generating subspace for $S$. An operator is  called
\textit{co-shift} if its adjoint is a unilateral shift.

 A contraction $A$ acting in a Hilbert space $\sH$ is
called \textit{completely non-unitary} if there is no nontrivial
reducing subspace of $A$, on which $A$ generates a unitary operator.
Given a contraction $A$ in $\sH$ then there is a canonical
orthogonal decomposition \cite [Theorem I.3.2]{SF}
\[
\sH=\sH_0\oplus \sH_1, \qquad A=A_0\oplus A_1, \quad A_j=A\uphar
\sH_j, \quad j=0,1,
\]
where $\sH_0$ and $\sH_1$ reduce $A$, the operator $A_0$ is a
completely non-unitary contraction, and $A_1$ is a unitary operator.
Moreover,
\[
\sH_1= \left(\bigcap\limits_{n\ge 1}\ker
D_{A^n}\right)\bigcap\left(\bigcap\limits_{n\ge 1}\ker
D_{A^{*n}}\right).
\]
Since
\[
\bigcap\limits_{k=0}^{n-1} \ker(D_{A}A^{k})=\ker D_{A^{n}},\;
\bigcap\limits_{k=0}^{n-1} \ker(D_{A^*}A^{*k})=\ker D_{A^{*n}},
\]
we get
\begin{equation}
\label{perp}
\begin{split}
&\bigcap\limits_{n\ge 1}\ker
D_{A^n}=\sH\ominus\cspan\left\{A^{*n}D_A \sH,\;
n=0,1,\ldots\right\},\\
&\bigcap\limits_{n\ge 1}\ker
D_{A^{*n}}=\sH\ominus\cspan\left\{A^nD_{A^*} \sH,\;
n=0,1,\ldots\right\}.
\end{split}
\end{equation}
It follows that
\begin{equation}
\label{cu}\begin{array}{l}
 A \;\mbox{is completely
non-unitary}\;\iff\left(\bigcap\limits_{n\ge 1}\ker
D_{A^n}\right)\bigcap\left(\bigcap\limits_{n\ge 1}\ker
D_{A^{*n}}\right)=\{0\}\iff\\
 \iff \cspan\{A^{*n}D_{A},\;A^mD_{A^*},\;n,m\ge
0\}=\sH.
\end{array}
\end{equation}
Note that
\[
\ker D_{A}\supset\ker D_{A^2}\supset\cdots\supset\ker
D_{A^n}\supset\cdots,
\]
\[
A\ker D_{A^n}\subset\ker D_{A^{n-1}},\; n=2,3,\ldots.
\]
From \eqref{perp} we get that the subspaces $\bigcap\limits_{n\ge
1}\ker D_{A^n}$ and $\bigcap\limits_{n\ge 1}\ker D_{A^{*n}}$ are
invariant with respect to $A$ and $A^*$, respectively, and
$A\uphar\bigcap\limits_{n\ge 1}\ker D_{A^n}$ and
$A^*\uphar\bigcap\limits_{n\ge 1}\ker D_{A^{*n}}$ are unilateral
shifts, moreover, these operators are the maximal unilateral shifts
contained in $A$ and $A^*$, respectively \cite[Theorem 1.1,
Corollary 1]{DR}. Thus, for a completely non-unitary contraction $A$
we have
\begin{equation}
\label{SHTCOSHT}
\begin{array}{l}
\bigcap\limits_{n\ge 1}\ker D_{A^n}=\{0\}\iff \quad A\quad\mbox{does
not contain a unilateral shift},\\
\bigcap\limits_{n\ge 1}\ker D_{A^{*n}}=\{0\}\iff \quad
A^*\quad\mbox{does not contain a unilateral shift}.
\end{array}
\end{equation}
By definition \cite{DR} the operator $A$ contains a co-shift $V$ if
the operator $A^*$ contains the unilateral shift $V^*$.

 The function (see \cite[Chapter VI]{SF})
\begin{equation}
\label{Char.funct.} \Phi_A(\lambda)=\left(-A+\lambda
D_{A^*}(I-\lambda A^*)^{-1}D_{A}\right)\uphar\sD_{A}
\end{equation}
is known as the Sz.-Nagy -- Foias \textit{characteristic function}
 of a contraction $A$ \cite{SF}. This function belongs to the Schur
class ${\bf S}(\sD_A,\sD_{A^*})$ and $\Theta_A(0)$ is a pure
contraction. The characteristic functions of $A$ and $A^*$ are
connected by the relation
\[
 \Phi_{A^*}(\lambda)=\Phi^*_A(\bar \lambda),\quad \lambda\in\dD.
\]

Two operator-valued functions $\Theta_1\in{\bf S}(\sM_1,\sN_1)$ and
$\Theta_2\in{\bf S}(\sM_2,\sN_2)$ \textit{coincide} \cite{SF} if
there are two unitary operators $V:\sN_1\to\sN_2$ and
$W:\sM_2\to\sM_1$ such that
\[
V\Theta_1(\lambda)W=\Theta_2(\lambda),\quad \lambda\in\dD.
\]
The result of Sz.-Nagy--Foias  \cite[Theorem VI.3.4]{SF} states that
two completely non-unitary contractions $A_1$ and $A_2$ are unitary
equivalent if and only if their characteristic functions
$\Phi_{A_1}$ and $\Phi_{A_2}$ coincide.

 It is well known that a function $\Theta(\lambda)$
from the Schur class ${\bf S}(\sM,\sN)$ has almost everywhere
non-tangential strong limit values $\Theta(\xi)$, $\xi \in \dT$,
where $\dT=\{\xi \in \dC: |\xi|=1\}$ stands for the unit circle; cf.
\cite{SF}. A function $\Theta \in {\bf S}(\sM,\sN)$ is called
\textit{inner} if $\Theta^*(\xi)\Theta(\xi)=I_\sM$ and
\textit{co-inner} if $\Theta(\xi)\Theta^*(\xi)=I_\sN$ almost
everywhere on $\xi\in\dT$. A function $\Theta\in {\bf S}(\sM,\sN)$
is called \textit{bi-inner}, if it is both inner and co-inner. A
contraction $T$ on a Hilbert space $\sH$ belongs to the class
$C_{0\,\cdot}$ ($C_{\cdot\, 0}$), if
\[
s-\lim\limits_{n\to\infty}A^n=0  \qquad
(s-\lim\limits_{n\to\infty}A^{*n}=0),
\]
respectively. By definition $C_{00}:=C_{0\,\cdot}\cap C_{\cdot\,
0}$. A completely non-unitary contraction $A$ belongs to the class
$C_{\cdot\,0}$, $C_{0\,\cdot}$, or $C_{00}$  if and only if its
characteristic function $\Phi_A(\lambda)$  is inner, co-inner, or
bi-inner, respectively (cf. \cite[Section VI.2]{SF}). Note that for
a completely non-unitary contraction $A$ the equality $\ker D_A=\ker
D_{A^*}\ne\{0\}$ is impossible because otherwise the subspace $\ker
D_A$ reduces $A$ and $A\uphar\ker D_A$ is a unitary operator.

We complete this section by a description of completely non-unitary
contractions with constant characteristic functions. Note that
 $\Phi_A(\lambda)=0\in {\bf S}(\{0\},\sD_{A^*})$ $\iff$ $A$
is a unilateral shift, and  $\Phi_A(\lambda)=0\in {\bf
S}(\sD_A,\{0\})$ $\iff$ $A$ is a co-shift.
\begin{theorem}
\label{const} Let $\sH$ be a separable Hilbert space. A completely
non-unitary contraction $A$ with nonzero defect operators has a
constant characteristic function if and only if $\sH$ is the
orthogonal sum
\[
\sH=\cH_1\oplus\cH_2
\]
and $A$ takes the operator matrix form
\begin{equation}
\label{coupl} A=\begin{bmatrix}S_1 &\Gamma\cr 0& S^*_2
\end{bmatrix}:\begin{array}{l}\cH_1\\\oplus\\\cH_2\end{array}\to
\begin{array}{l}\cH_1\\\oplus\\\cH_2\end{array},
\end{equation}
where $S_1$ and $S_2$ are unilateral shifts in $\cH_1$ and $\cH_2$,
respectively, and $\Gamma$ is a contraction such that
\begin{equation}
\label{Ga} \left\{ \begin{array}{l} \ran\Gamma\subset \sD_{S^*_1},
 \ran\Gamma^*\subset \sD_{S^*_2},\\
||\Gamma f||<||f||,\;
f\in \sD_{S^*_2}\setminus\{0\},\\
||\Gamma^*h||<||h||,\; h\in \sD_{S^*_1}\setminus\{0\}.
\end{array}\right.
\end{equation}
In particular, the characteristic function of $A$ is identically
equal zero if and only if $A$ is the orthogonal sum of a shift and
co-shift.
\end{theorem}
\begin{proof}
Suppose that the contraction $A$ takes the form \eqref{coupl} with
unilateral shifts $S_1$ and $S_2$, and the contraction $\Gamma$ with
the properties \eqref{Ga}. Then
\begin{equation}
\label{DA} D^2_A=\begin{bmatrix} 0&0\cr 0& D_{S^*_2}-\Gamma^*\Gamma
\end{bmatrix}:\begin{array}{l}\cH_1\\\oplus\\\cH_2\end{array}\to
\begin{array}{l}\cH_1\\\oplus\\\cH_2\end{array},
\end{equation}
and
\begin{equation}
\label{DA*} D^2_{A^*}=\begin{bmatrix} D_{S^*_1}-\Gamma\Gamma^*&0\cr
0& 0
\end{bmatrix}:\begin{array}{l}\cH_1\\\oplus\\\cH_2\end{array}\to
\begin{array}{l}\cH_1\\\oplus\\\cH_2\end{array}.
\end{equation}
Since $\sD_{S^*_1}=\ker S^*_1$, $\sD_{S^*_2}=\ker S^*_2$, and
$D_{S^*_1}$ and $D_{S^*_2}$ are the orthogonal projections in $\sH$
onto $\sD_{S^*_1}$ and $\sD_{S^*_2}$, respectively, we get from
\eqref{Ga} the relations
\begin{equation}
\label{DG} \sD_A=\sD_{S^*_2},\; \sD_{A^*}=\sD_{S^*_1}.
\end{equation}
 Taking into
account that $\cH_2$ is an invariant subspace for $A^*$, we have
\[
D_{A^*}(I_\sH-\lambda A^*)^{-1}D_{A}=0.
\]
Hence $\Phi_A(\lambda)=\Gamma\uphar\sD_{S^*_2}=const.$

Because $S_1$ and $S_2$ are unilateral shifts, we get
\[
\cH_1=\sum\limits_{n\ge 0}\oplus S^n_1\sD_{S^*_1},\;
\cH_2=\sum\limits_{n\ge 0}\oplus S^n_2\sD_{S^*_2}.
\]
Since $\sH=\cH_1\oplus\cH_2$, the operator $A$ is completely
non-unitary. If $\Gamma=0$ then $A$ is the orthogonal sum of a shift
and co-shift.

 Now suppose
that the characteristic function of $A$ is a constant. From
\eqref{Char.funct.} we get
\[
D_{A^*}A^{*n}D_{A}=0,\; D_{A}A^{n}D_{A^*}=0,\; n=0,1,2,\ldots.
\]
It follows
\[
\begin{array}{l}
\cspan\{D_{A^{*n}}\sD_{A},\; n=0,1,\ldots\}\subset\ker D_{A^*}\iff
\bigcap\limits_{n\ge 1}\ker D_{A^n}\supset \sD_{A^*},\\
\cspan\{D_{A^{n}}\sD_{A^*},\; n=0,1,\ldots\}\subset\ker D_{A}\iff
\bigcap\limits_{n\ge 1}\ker D_{A^{*n}}\supset\sD_A.
\end{array}
\]
Let
\[
\cH_1=\bigcap\limits_{n\ge 1}\ker D_{A^n},\;\cH_2=
\bigcap\limits_{n\ge 1}\ker D_{A^{*n}}.
\]
Since
\[
 A\cH_1\subset \cH_1\quad \mbox{and}\quad
A\cH_1\perp\sD_{A^*},
\]
we get $\cH_1\ominus A\cH_1\supset \sD_{A^*}$ and similarly
$\cH_2\ominus A^*\cH_2\supset \sD_{A}$. Let $h\in\cH_1$ and $h\perp
\sD_{A^*}$. It follows
\[
h\in\ker D_{A^*}\bigcap\left(\bigcap\limits_{n\ge 1}\ker
D_{A^n}\right).
\]
Then $h=Ag$, $g\in\ker D_A$. Hence $g\in \bigcap\limits_{n\ge 1}\ker
D_{A^n}=\cH_1$, i.e., $\cH_1\ominus A\cH_1=\sD_{A^*}$. Similarly
$\cH_2\ominus A^*\cH_2=\sD_{A}$.

Since $A$ is completely non-unitary contraction, the operators
$A\uphar\cH_1$ and $A^*\uphar\cH_2$ are unilateral shifts. Therefore
\begin{equation}
\label{shift}
 \cH_1=\sum\limits_{n=0}^\infty\oplus A^n\sD_{A^*},\;
\cH_2=\sum\limits_{n=0}^\infty\oplus A^{*n}\sD_A.
\end{equation}
Note that for all $\f,\psi\in \sH$
\[
(A^mD_{A^*}\f,A^{*k}D_{A}\psi)=(D_{A}A^{m+k}D_{A^*}\f,\psi)=0,\;m,k=0,1,2\ldots.
\]
Hence $\cH_1\perp\cH_2$. Taking into account \eqref{shift} and the
relation \[
 \sH\ominus\cH_1=\cspan\{A^{*n}\sD_A,\;
n=0,1,2\ldots\},
\]
we get $\sH\ominus\cH_1=\cH_2$. Because $\cH_1$ is invariant with
respect to $A$, the matrix form of $A$ is of the form \eqref{coupl}
with unilateral shifts
\[
S_1:=A\uphar \cH_1,\; S_2:=A^*\uphar \cH_2,
\]
and some operator $\Gamma\in\bl(\cH_2,\cH_1$. Since $A$ is a
contraction, we have
\[
\begin{array}{l}
||\Gamma f||^2\le ||D_{S_2^*}f||^2,\; f\in \cH_2,\\
||\Gamma^* h||^2\le ||D_{S_1^*}h||^2,\; h\in \cH_1.
\end{array}
\]
From \eqref{DA} and \eqref{DA*} we get
\[
\cran (D_{S_2^*}-\Gamma^*\Gamma)=\sD_{A},\; \cran
(D_{S_1^*}-\Gamma\Gamma^*)=\sD_{A^*}.
\]
It follows that \eqref{Ga} holds true and $\Phi_A(\lambda)=\Gamma$.

If $A$ is the orthogonal sum of a shift and co-shift then clearly
the characteristic function of $A$ is identically zero.
 \end{proof}

\section{Contractions generated by a contraction}
In this section we define and study the subspaces and the
corresponding operators obtained from a completely non-unitary
contraction $A$ in a separable Hilbert space $\sH$.

Suppose $\ker D_A\ne \{0\}$. Define the subspaces
\begin{equation}
\label{hnm} \left\{
\begin{array}{l}
\sH_{0,0}:=\sH\\
 \sH_{n,0}=\ker D_{A^n},\;
  \sH_{0,m}:=\ker
D_{A^{*m}},\\
\sH_{n, m}:=\ker D_{A^{n}}\cap \ker D_{A^{*m}},\;
m,n\in\dN
\end{array}
\right.
\end{equation}
Let $P_{n,m}$ be the orthogonal projection in $\sH$ onto
$\sH_{n,m}$. Define the contractions
\begin{equation}
\label{Anm}
 A_{n, m}:=P_{n,
m}A\uphar\sH_{n, m}\in\bL(\sH_{n,m})
\end{equation}
and
\begin{equation}\label{CA}
\cA_{n,m}:=A_{n,m}P_{n+1,m}\uphar\sH_{n,m}\in\bL(\sH_{n,m}).
\end{equation}
In the next theorem we establish the main properties of $A_{n,m}$
and $\cA_{n,m}$.
\begin{theorem}
\label{RELATT}
\begin{enumerate}
\item Hold the relations
\begin{equation}
\label{AKNM}\left\{\begin{array}{l}
 \ker D_{A^k_{n,m}}=\sH_{n+k,m}\\
 \ker D_{A^{*k}_{n,m}}=\sH_{n,m+k}
 \end{array},\right. k=1,2,\ldots,
\end{equation}
\begin{equation}
\label{REL2} \left\{\begin{array}{l}
\sD_{A_{n,m}}=\cran{(P_{n,m}D_{A^{n+1}})},\\
\sD_{A^*_{n,m}}=\cran{(P_{n,m}D_{A^{*{m+1}}})}
\end{array}\right.,
\end{equation}
\begin{equation}
\label{REL} \left\{ \begin{array}{l} A\sH_{n,m}=\sH_{n-1,
m+1},\;n\ge 1,\\
 A^*\sH_{n,m}=\sH_{n+1, m-1},\; m\ge 1
 \end{array}\right.,
\end{equation}
\begin{equation}
\label{RELL3} \left\{ \begin{array}{l} \ker
D_{\cA_{n,m}^k}=\sH_{n+k,m}\\
 \ker D_{\cA_{n,m}^{*k}}=\sH_{n,m+k}\end{array}\right. k=1,2,\ldots,
\end{equation}
\begin{equation}
\label{RELL4} \left\{\begin{array}{l}
\sD_{\cA_{n,m}}=\sD_{A_{n+1,m}}\\
\sD_{\cA^*_{n,m}}=\sD_{A^*_{n+1,m}}\end{array} \right.,
\end{equation}
\begin{equation}
\label{UE1}
 \left(A_{n,m}\right)_{k,l}=A_{n+k,m+l}.
\end{equation}
\item
 The operators $\left\{A_{n,m}\right\}$ and $\left\{\cA_{n,m}\right\}$ are completely non-unitary
contractions.
\item
 The operators
 \[
A_{n,0},\; A_{n-1,1},\; \ldots, A_{n-k,k},\ldots,A_{0,n}
\]
are unitarily equivalent and
\begin{equation}
\label{UE}  A_{n-1,m+1}Af=AA_{n,m}f,\;f\in\sH_{n,m},\; n\ge 1.
\end{equation}
\item
 The operators
 \[
\cA_{n,0},\; \cA_{n-1,1},\; \ldots, \cA_{n-k,k},\ldots,\cA_{0,n}
\]
are unitarily equivalent and
\begin{equation}
\label{UE2} \cA_{n-1,m+1} Af=A\cA_{n,m}f,\;f\in\sH_{n,m}, \; n\ge 1.
\end{equation}
\item The following statements are equivalent
\begin{enumerate}
\def\labelenumi{\rm (\roman{enumi})}
\item $\cA_{n,0}\in C_{\cdot\,0}$ $(\cA_{n,0}\in C_{0\,\cdot})$ for some $n$,
\item $A_{n+1-k,k}\in C_{\cdot\,0}$ $(A_{n+1-k,k}\in
C_{0\,\cdot})$ for all $k=0,1,\ldots,n+1. $
\end{enumerate}
\end{enumerate}
\end{theorem}
\begin{proof} It is sufficient to prove the first equality from \eqref{AKNM}.
 From \eqref{hnm} and \eqref{Anm} we have
\[
\begin{array}{l}
f\in\sH_{n,m},\;f\in\ker D_{A^k_{n,m}}\iff \left\{
\begin{array}{l}||f||=||A^n f||=||A^{*m}f||\\
||f||=||A^k_{n,m}f||
\end{array}
 \right.\\
 \iff Af,\ldots, A^kf\in\sH_{n,m}\iff f\in \sH_{n+k,m}.
 \end{array}
\]
This proves \eqref{AKNM}. Hence
\[
\begin{array}{l}
\sD_{A_{n,m}}=\sH_{n,m}\ominus\sH_{n+1,m}=\sH_{n,m}\ominus \left(\ker D_{A^{n+1}}\cap \ker D_{A^{*m}}\right)=\\
=\sH_{n,m}\cap\overline{\sD_{A^{n+1}}+\sD_{A^{*m}}}=\cran{(P_{n,m}D_{A^{n+1}})},\\
\sD_{A^*_{n,m}}=\sH_{n,m}\ominus\sH_{n,m+1}=\sH_{n,m}\ominus
\left(\ker D_{A^{n}}\cap \ker
D_{A^{*{m+1}}}\right)=\\
=\sH_{n,m}\cap\overline{\sD_{A^{n}}+\sD_{A^{*{m+1}}}}=\cran{(P_{n,m}D_{A^{*{m+1}}})},
\end{array}
\]
i.e., relations \eqref{REL2} are valid.
 Furthermore if $n\ge 2$ then
\[
f\in\sH_{n,m}\iff \left\{\begin{array}{l}Af\in\ker
D_{A^{n-1}},\\A^*Af=f,\\
f\in\ker D_{A^{*m}}\;(\mbox{for}\; m\ge 1)\end{array}\right. \iff
Af\in\ker D_{A^{n-1}}\cap\ker D_{A^{*{m+1}}} =\sH_{n-1,m+1}.
\]
If $n=1$ then
\[
f\in\sH_{1,m}\iff  \left\{\begin{array}{l}A^*Af=f,\\
f\in\ker D_{A^{*m}}\end{array}\right. \iff Af\in\ker D_{A^{*{m+1}}}
=\sH_{0,m+1}.
\]

 Similarly $ A^*\sH_{n,m}=\sH_{n+1, m-1},\; m\ge 1.$ Therefore
 relations \eqref{REL} hold true.

 Let $\f\in\sH$, $\psi\in\sH_{n-1,m+1}$. Then $A^*\psi\in\sH_{n,m}$
and
\[
(AP_{n,m}\f,\psi)=(P_{n,m}\f, A^*\psi)=(\f,
A^*\psi)=(A\f,\psi)=(P_{n-1,m+1}A\f,\psi).
\]
Hence
\begin{equation}
\label{IMPORT}
 AP_{n,m}=P_{n-1,m+1}A.
\end{equation}
 Taking into account \eqref{REL}, we get
\[
AP_{n,m}Ah=P_{n-1,m+1}A Ah,\; h\in\sH_{n,m}.
\]
This proves \eqref{UE}. Since $A$ isometrically maps $\sH_{n,m}$
onto $\sH_{n-1,m+1}$ for $n\ge 1$, the operators $A_{n-1,m+1}$ and
$A_{n,m}$ are unitarily equivalent, and therefore the operators
\[
A_{n,0},\; A_{n-1,1},\; \ldots, A_{n-k,k},\ldots,A_{0,n}
\]
are unitarily equivalent.

Note that \eqref{AKNM} and \eqref{REL} yield the equalities
\begin{equation}
\label{shtcosht}
\begin{array}{l}
\bigcap\limits_{k\ge1}\ker D_{A^k_{n,m}}=\ker
D_{A^{*m}}\bigcap\left(\bigcap\limits_{j\ge 1}\ker D_{A^j}\right)=
A^m\left(\bigcap\limits_{j\ge 1}\ker D_{A^j}\right),\\
\bigcap\limits_{k\ge1}\ker D_{A^{*k}_{n,m}}=\ker
D_{A^{n}}\bigcap\left(\bigcap\limits_{j\ge 1}\ker
D_{A^{*j}}\right)=A^{*n}\left(\bigcap\limits_{j\ge 1}\ker
D_{A^{*j}}\right).
\end{array}
\end{equation}
Since $A$ is a completely non-unitary contraction, we get
\[
\left(\bigcap\limits_{k\ge1}\ker D_{A^k_{n,m}}\right)\bigcap
\left(\bigcap\limits_{k\ge1}\ker D_{A^{*k}_{n,m}}\right)=\{0\}.
\]
It follows that the contractions $A_{n,m}$ are completely
non-unitary.

Note that $\sH_{n-1,m+1}\subset\sH_{n-1,m}$ and
$\sH_{n+1,m}\subset\sH_{n,m}$. Using \eqref{REL} we get
\[
\begin{array}{l}
A_{n-1,m+1}P_{n,m+1}=P_{n-1,m+1}AP_{n,m+1}=AP_{n,m+1},\\
A_{n,m}P_{n+1,m}=P_{n,m}AP_{n+1,m}=AP_{n+1,m}.
\end{array}
\]
In particular, it follows that the operators $A_{n,m}P_{n+1,m}$ are
partial isometries. From \eqref{IMPORT} we obtain
\[
AP_{n,m+1}A=A^2 P_{n+1,m},
\]
i.e.,
\[
A_{n-1,m+1}P_{n,m+1}Af=AA_{n,m}P_{n+1,m}f\quad\mbox{for all}\quad
f\in\sH_{n,m}.
\]
Because $A$ is unitary operator from $\sH_{n,m}$ onto
$\sH_{n-1,m+1}$, we get \eqref{UE2} and so the operators
$\cA_{n-1,m+1}$ and $\cA_{n,m}$ are unitarily equivalent.

By induction it can be easily proved that for every $k\in\dN$ holds
the equality
\begin{equation}
\label{C00}
\cA_{n,m}^kf=(AP_{n+1,m})^kf=AA^{k-1}_{n+1,m}P_{n+1,m}f,\;
f\in\sH_{n,m}.
\end{equation}
Since $A\uphar\sH_{n+1,m}$ is isometric, relations \eqref{C00} imply
\[
||\cA_{n,m}^kf||=||A^{k-1}_{n+1,m}P_{n+1,m}f||,\; f\in\sH_{n,m}, \;
k\in\dN.
\]
It follows the equivalence of the statements (a) and (b) and
\[
\ker D_{\cA_{n,m}^k}=\ker D_{A^{k-1}_{n+1,m}}=\sH_{n+k,m}.
\]
Similarly, since $(A_{n,m}P_{n+1,m})^*=A^*_{n,m}P_{n,m+1}$, we get
\[
\ker D_{\cA_{n,m}^{*k}}=\ker D_{A^{*k-1}_{n,m+1}}=\sH_{n,m+k}.
\]
Thus, relations \eqref{RELL3} are valid.

Now we get that the operators $A_{n,m}P_{n+1,m}$ are completely
non-unitary.
From \eqref{AKNM} it follows that
\[
\begin{array}{l}
\ker D_{A^k_{n,m}}\cap\ker
D_{A^{*l}_{n,m}}=\sH_{n+k,m}\cap\sH_{n,m+l}=\\
\ker D_{A^{n+k}}\cap\ker D_{A^{*m}}\cap\ker D_{A^n}\cap\ker
D_{A^{*m+l}}=\ker D_{A^{n+k}}\cap \ker D_{A^{*m+l}}=\sH_{n+k,m+l}.
\end{array}
\]
Hence
\[
\left(A_{n,m}\right)_{k,l}=P_{n+k,m+l}P_{n,m}A\uphar\sH_{n+k,m+l}=A_{n+k,m+l}.
\]
\end{proof}
 The relation \eqref{UE1} yields the following picture
for the creation of the operators $A_{n,m}$:

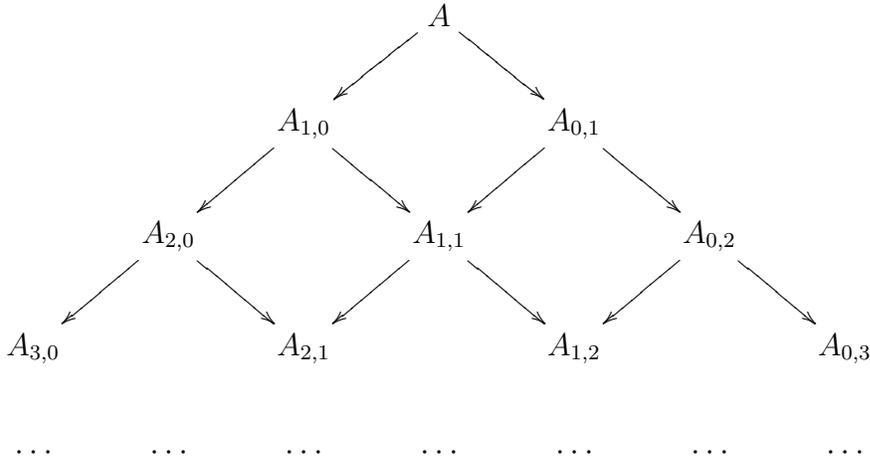
\xymatrix{
&&&A\ar[ld]\ar[rd]\\
&&A_{1,0}\ar[ld]\ar[rd]&&A_{0,1}\ar[ld]\ar[rd]\\
&A_{2,0}\ar[ld]\ar[rd]&&A_{1,1}\ar[ld]\ar[rd]&&A_{0,2}\ar[ld]\ar[rd]\\
A_{3,0}&&
A_{2,1}&&A_{1,2}&&A_{0,3}\\
\cdots&\cdots&\cdots&\cdots&\cdots&\cdots&\cdots}

The process terminates on the $N$-th step if and only if
\[
\begin{array}{l}
\ker D_{A^N}=\{0\}\iff \ker D_{A^{N-1}}\cap\ker
D_{A^*}=\{0\}\iff\ldots\\ \ker D_{A^{N-k}}\cap\ker
D_{A^{*k}}=\{0\}\iff\ldots \ker D_{A^{*N}}=\{0\}.
\end{array}
\]
Note that from \eqref{SHTCOSHT}, \eqref{RELL3}, and \eqref{shtcosht}
we get
\begin{proposition}
\label{ShtCOsht} Let $A$ be a completely non-unitary contraction. If $A$ does not contain a unilateral shift (co-shift) then the same is true for the operators $\cA_{n,m}$ and $A_{n,m}$ for all $n$ and $m$.
Conversely, if for some $n$ and $m$ the operator $\cA_{n,m}$ or $A_{n,m}$ does not contain a unilateral shift (co-shift) then the same is valid for $A$.
\end{proposition}

 Let $\delta_A=\dim\sD_A$,
$\delta_{A^*}=\dim\sD_{A^*}$ be the defect numbers of a completely
non-unitary contraction $A$. For $n=1,\ldots$ denote by $\delta_{n}$
and $\delta^*_{n}$ the defect numbers of unitarily equivalent
operators $\{A_{n-m,m}\}_{m=0}^n$. From the relations \eqref{REL2}
it follows that
\[\begin{array}{l}
\delta_n=\dim\sD_{A_{0,n}}=\dim\left(\cran(P_{0,n}D_{A})\right)=\dim\left(\sD_A\ominus(\sD_{A}\cap\sD_{A^{*n}})\right),\\
\delta^*_n=\dim\sD_{A^*_{n,0}}=\dim\left(\cran(P_{n,0}D_{A^*})\right)=\dim\left(\sD_{A^*}\ominus(\sD_{A^*}\cap\sD_{A^{n}})\right).
\end{array}
\]
Thus
\[
\begin{array}{l}
 \delta_A\ge\delta_1\ge\cdots\ge\delta_n\ge\cdots,\\
\delta_{A^*}\ge\delta^*_1\ge\cdots\ge\delta^*_n\ge\cdots.
\end{array}
\]
Observe also that
\[
\delta_1=\dim\left(\sD_A\ominus(\sD_{A}\cap\sD_{A^{*}})\right),\;\delta^*_1=\dim\left(\sD_{A^*}\ominus(\sD_{A}\cap\sD_{A^{*}})\right),
\]
and by induction
\[
\delta_n=\dim\left(\sD_{A_{n-1,0}}\ominus(\sD_{A_{n-1,0}}\cap\sD_{A^*_{n-1,0}})\right),\;
\delta^*_n=\dim\left(\sD_{A^*_{n-1,0}}\ominus(\sD_{A_{n-1,0}}\cap\sD_{A^*_{n-1,0}})\right).
\]

\section{Passive discrete-time linear systems and their
transfer functions}  \label{secS}
\subsection{Basic definitions} Let
$\sM,\sN$, and $\sH$ be separable Hilbert spaces. A linear system
$\tau=\left\{\begin{bmatrix} D&C \cr B&A\end{bmatrix};\sM,\sN,
\sH\right\}$ with bounded linear operators $A$, $B$, $C$, $D$  of
the form
\begin{equation}
\label{passive}
 \left\{
 \begin{array}{l}
  h_{k+1}=Ah_k+B\xi_k,\\
  \sigma_k=Ch_k+D\xi_k,
\end{array}
\right. \qquad k\ge 0,
\end{equation}
where $\{h_k\}\subset \sH$, $\{\xi_k\}\subset \sM$,
$\{\sigma_k\}\subset \sN$, is called a \textit{discrete-time
system}. The Hilbert spaces $\sM$ and $\sN$ are called the input and
the output spaces, respectively, and the Hilbert space $\sH$ is
called the state space. The operators $A$, $B$, $C$, and $D$ are
called the \textit{state space operator}, the \textit{control
operator}, the \textit{observation operator}, and the
\textit{feedthrough operator} of $\tau$, respectively. If the linear
operator $T_\tau$ defined by the block form
\begin{equation}
\label{sys} T_\tau=\begin{bmatrix} D&C \cr B&A\end{bmatrix} :
\begin{array}{l} \sM \\\oplus\\ \sH \end{array} \to
\begin{array}{l} \sN \\\oplus\\ \sH \end{array}
\end{equation}
is contractive, then the corresponding discrete-time system is said
to be \textit{passive}. If the block operator matrix $T_\tau$ is
isometric (co-isometric, unitary), then the system is said to be
\textit{isometric} (\textit{co-isometric}, \textit{conservative}).
Isometric and co-isometric systems were studied by L.~de Branges and
J.~Rovnyak (see \cite{BrR1}, \cite{BrR2}) and by T.~Ando (see
\cite{Ando}), conservative systems have been investigated by
B.~Sz.-Nagy and C.~Foia\c{s} (see \cite{SF}) and M.S.~Brodski\u{\i}
(see \cite{Br1}). Passive systems have been studied by D.Z.~Arov et
al (see \cite{A}, \cite{Arov}, \cite{ArKaaP}, \cite{ArNu1},
\cite{ArNu2}).

The subspaces
\begin{equation}
\label{CO} \sH^c:=\cspan\{A^{n}B\sM:\,n=0,1,\ldots\} \mbox{ and }
\sH^o=\cspan\{A^{*n}C^*\sN:\,n=0,1,\ldots\}
\end{equation}
are said to be the \textit{controllable} and \textit{observable}
subspaces of the system $\tau$, respectively. The system $\tau$ is
said to be \textit{controllable} (\textit{observable}) if
$\sH^c=\sH$ ($\sH^o=\sH$), and it is called \textit{minimal} if
$\tau$ is both controllable and observable. The system $\tau$ is
said to be \textit{simple} if $\sH=\clos \{\sH^c+\sH^o\}$ (the
closure of the span). It follows from \eqref{CO} that
\begin{equation}
\label{eqco}
 (\sH^c)^\perp=\bigcap\limits_{n=0}^\infty\ker(B^*A^{*n}),\quad
 (\sH^o)^\perp=\bigcap\limits_{n=0}^\infty\ker(CA^{n}),
\end{equation}
and therefore there are the following alternative characterizations:
\begin{enumerate}\def\labelenumi{\rm (\alph{enumi})}
\item $\tau$ is controllable $\iff
\;\bigcap\limits_{n=0}^\infty\ker(B^*A^{*n})=\{0\}$;
\item  $\tau$ is observable $\iff
\;\bigcap\limits_{n=0}^\infty\ker(CA^{n})=\{0\}$;
\item  $\tau$ is simple $\iff
\left(\bigcap\limits_{n=0}^\infty\ker(B^*A^{*n})\right)\cap
\left(\bigcap\limits_{n=0}^\infty\ker(CA^{n})\right)=\{0\}.$
\end{enumerate}
The \textit{transfer function}
\begin{equation}
\label{TrFu} \Theta_\tau(\lambda):=D+\lambda C(I_{\sH}-\lambda
A)^{-1}B, \quad \lambda \in \dD,
\end{equation}
of the passive system $\tau$ belongs to the Schur class ${\bf
S}(\sM,\sN)$ \cite{A}. Conservative systems are also called the
\textit{unitary colligations}
 and their transfer
functions are called the characteristic functions \cite{Br1}.

The examples of conservative systems are given by
\[
\Sigma=\left\{\begin{bmatrix}-A& D_{A^*}\cr D_{A}& A^*\end{bmatrix};
 \sD_{A},\sD_{A^*},\sH\right\},\;\Sigma_*=\left\{\begin{bmatrix}-A^*&
D_{A}\cr D_{A^*}& A\end{bmatrix};  \sD_{A^*},\sD_{A},\sH\right\}.
\]
The transfer functions of these systems
\[
\Phi_\Sigma(\lambda)=\left(-A+\lambda D_{A^*}(I_\sH-\lambda
A^*)^{-1}D_{A}\right)\uphar\sD_{A}, \quad \lambda \in \dD
\]
and
\[
\Phi_{\Sigma_*}(\lambda)=\left(-A^*+\lambda D_{A}(I_\sH-\lambda
A)^{-1}D_{A^*}\right)\uphar\sD_{A^*}, \quad \lambda \in \dD
\]
are exactly characteristic functions of $A$ and $A^*$,
correspondingly.

It is well known that every operator-valued function
$\Theta(\lambda)$ from the Schur class ${\bf S}(\sM,\sN)$ can be
realized as the transfer function of some passive system, which can
be chosen as controllable isometric (observable co-isometric, simple
conservative, minimal passive); cf. \cite{BrR2}, \cite{SF},
\cite{Ando} \cite{A}, \cite{ArKaaP}, \cite{ADRS}. Moreover, two
controllable isometric (observable co-isometric, simple
conservative) systems with the same transfer function are unitarily
similar: two discrete-time systems
\[
\tau_1=\left\{\begin{bmatrix} D&C_1 \cr
B_1&A_1\end{bmatrix};\sM,\sN,\sH_{1}\right\} \quad \mbox{and} \quad
\tau_2=\left\{\begin{bmatrix} D&C_2 \cr
B_2&A_2\end{bmatrix};\sM,\sN,\sH_{2}\right\}
\]
are said to be \textit{unitarily similar} if there exists a unitary
operator $U$ from $\sH_{1}$ onto $\sH_{2}$ such that
\[
A_1 =U^{-1}A_2U,\quad B_1=U^{-1}B_2,\quad C_1=C_2 U;
\]
cf. \cite{BrR1}, \cite{BrR2}, \cite{Ando}, \cite{Br1}, \cite{ADRS}.
However, a result of D.Z.~Arov~\cite{A} states that two minimal
passive systems $\tau_1$ and $\tau_2$ with the same transfer
function $\Theta(\lambda)$ are only \textit{weakly similar}, i.e.,
there is a closed densely defined operator $Z:\sH_{1}\to\sH_{2}$
such that $Z$ is invertible, $Z^{-1}$ is densely defined, and
\[
Z A_1f =A_2 Zf, \quad C_1f=C_2 Zf,\quad f\in\dom Z,\quad \mbox{and}
\quad ZB_1=B_2.
\]

\subsection{Defect functions of the Schur class functions}

The following result \cite[Proposition V.4.2]{SF} is needed in the
sequel.
\begin{theorem}\label{Fact}
Let $\sM$ be a separable Hilbert space and let $N(\xi)$,
$\xi\in\dT$,  be an $\bL(\sM)$-valued measurable function such that
$0\le N(\xi)\le I_\sM$. Then there exist a Hilbert space $\sK$ and
an outer function $\varphi(\lambda) \in {\bf S}(\sM,\sK)$ satisfying
the following conditions:
\begin{enumerate}
\def\labelenumi{\rm (\roman{enumi})}
\item $\varphi^*(\xi)\varphi(\xi)\le N^2(\xi)$ a.e.
on $\dT$;

\item if $\wt \sK$ is a Hilbert space and $\wt\varphi(\lambda)\in
{\bf S}(\sM,\wt\sK)$ is such that
$\wt\varphi^*(\xi)\wt\varphi(\xi)\le N^2(\xi)$ a.e.  on $\dT$, then
$\wt\varphi^*(\xi)\wt\varphi(\xi)\le \varphi^*(\xi)\varphi(\xi)$
a.e.  on $\dT$.
\end{enumerate}
Moreover, the function $\varphi(\lambda)$ is uniquely defined up to
a left constant unitary factor.
\end{theorem}

Assume that $\Theta(\lambda)\in {\bf S}(\sM,\sN)$ and denote by
$\varphi_\Theta(\xi)$ and $\psi_\Theta(\xi)$, $\xi \in \dT$ the
outer functions which are solutions of the factorization problem
described in Theorem \ref{Fact} for
$N^2(\xi)=I_\sM-\Theta^*(\xi)\Theta(\xi)$ and
$N^2(\bar\xi)=I_\sN-\Theta(\bar\xi)\Theta^*(\bar\xi)$, respectively.
Clearly, if $\Theta(\lambda)$ is inner or co-inner, then
$\varphi_\Theta=0$ or $\psi_\Theta=0$, respectively.  The functions
$\f_\Theta(\lambda)$ and $\psi_\Theta(\lambda)$ are called the right
and left \textit{defect functions} (or the \textit{spectral
factors}), respectively, associated with $\Theta(\lambda)$; cf.
\cite{BC}, \cite{BD}, \cite{BDFK1}, \cite{BDFK2}, \cite{DR}.
 The following result has been established in  \cite[Theorem 1.1, Corollary 1]{DR}
(see also \cite[Theorem 3]{BDFK1}, \cite[Theorem 1.5]{BDFK2}).
\begin{theorem}
\label{DBR} Let $\Theta(\lambda)\in {\bf S}(\sM,\sN)$ and let
\[
\tau=\left\{\begin{bmatrix} D&C \cr
B&A\end{bmatrix};\sM,\sN,\sH\right\}
\]
be a simple conservative system with transfer function $\Theta$.
Then
\begin{enumerate}
\item the functions $\varphi_\Theta(\lambda)$ and $\psi_\Theta(\lambda)$
take the form
\[
\begin{split}
& \varphi_\Theta(\lambda)=P_\Omega(I_\sH-\lambda A)^{-1}B,\\
&\psi_\Theta(\lambda)=C(I_\sH-\lambda A)^{-1}\uphar\Omega_*,
\end{split}
\]
where
\[
\Omega=(\sH^o)^\perp\ominus A(\sH^o)^\perp ,\;
\Omega_*=(\sH^c)^\perp\ominus A^*(\sH^c)^\perp
\]
and $P_\Omega$ is the orthogonal projector from $\sH$ onto $\Omega$;
\item
$\varphi_\Theta(\lambda)=0$ ($\psi_\Theta(\lambda)=0$) if and only
if the system $\tau$ is observable (controllable).
\end{enumerate}
\end{theorem}
The defect functions play an essential role in the problems of the
system theory, in particular, in the problem of similarity and
unitary similarity of the minimal passive systems with equal
transfer functions \cite{ArNu1}, \cite{ArNu2} and in the problem of
\textit{optimal} and $(*)$ \textit{optimal} realizations of the
Schur function \cite{Arov}, \cite{ArKaaP}.

\subsection{Parametrization of contractive block-operator matrices}
Let $\sH,\,\sK,$ $\sM$ and $\sN$ be Hilbert spaces. The following
theorem goes back to \cite{AG}, \cite{DaKaWe}, \cite {ShYa}; other
proofs of the theorem can be found in \cite{KolMal}, \cite{Mal2},
\cite{ARL}, \cite{AHS1}.
\begin{theorem} \label{ParContr1}
Let $A \in \bL(\sH,\sK)$, $B \in \bL(\sM,\sK)$, $C \in
\bL(\sH,\sN)$, and $D \in \bL(\sM,\sN)$. The operator matrix
\[
T= \begin{bmatrix} D&C\cr B&A \end{bmatrix}:
\begin{array}{l}\sM\\\oplus\\\sH\end{array}\to
\begin{array}{l}\sN\\\oplus\\\sK\end{array}
\]
is a contraction if and only if $T$ is of the form
\begin{equation}\label{twee}
T=
\begin{bmatrix} -KA^*M+D_{K^*}XD_{M} &KD_{A} \cr
D_{A^*}M&A\end{bmatrix},
\end{equation}
where $A \in \bL(\sH,\sK)$, $M\in\bL(\sM,\sD_{A^*})$,
$K\in\bL(\sD_{A},\sN)$, and $X\in\bL(\sD_{M},\sD_{K^*})$ are
contractions, all uniquely determined by $T$. Furthermore, the
following equality holds for all $h \in \sM$, $f \in \sH$:
\begin{equation} \label{Contr}
\begin{gathered}
\left\|\begin{bmatrix}h\cr f\end{bmatrix}\right\|^2
-\left\|\begin{bmatrix} -KA^*M+D_{K^*}XD_{M} &KD_{A} \cr
D_{A^*}M&A\end{bmatrix}
\begin{bmatrix}h\cr f\end{bmatrix}  \right\|^2 \\
=\|D_K(D_{A}f-A^*Mh)-K^*XD_{M}h\|^2+\|D_{X}D_{M}h\|^2.
\end{gathered}
\end{equation}
\end{theorem}
\begin{corollary}
\label{iscois1} Let
\[
T=
\begin{bmatrix}-KA^*M+D_{K^*}XD_{M} &KD_{A} \cr
D_{A^*}M&A\end{bmatrix}:
\begin{array}{l}\sM\\\oplus\\\sH\end{array}\to
\begin{array}{l}\sN\\\oplus\\\sK\end{array}
\]
be a contraction. Then
\begin{enumerate}
\item$T$ is isometric if
and only if
\[
D_KD_A=0,\;D_XD_M=0,
\]
\item
$T$ is co-isometric if and only if
\[
D_{M^*}D_{A^*}=0,\; D_{X^*}D_{K^*}=0.
\]
\end{enumerate}
\end{corollary}
Note that the relation $D_YD_Z=0$ for contractions $Y$ and $Z$ means
that either $Z$ is an isometry and $Y=0$ or $\sD_Z\ne\{0\}$ and $Y$
is an isometry. From \eqref{Contr} we get the following statement

\textit{If $T$ given by \eqref{twee} is unitary then $D_{K^*}=0$
$\iff$ $D_{M}=0$.}

Let $\tau=\left\{\begin{bmatrix} D&C \cr
B&A\end{bmatrix};\sM,\sN,\sH\right\}$ be a conservative system. Then
from Corollary \ref{iscois1} we get
\begin{equation}
\label{SP}
\begin{split}
&(\sH^{c})^\perp=\bigcap\limits_{n\ge
0}\ker(D_{A^*}A^{*n})=\bigcap\limits_{n\ge 1}\ker(D_{A^{*n}}),\\
& (\sH^{o})^\perp=\bigcap\limits_{n\ge
0}\ker(D_{A}A^n)=\bigcap\limits_{n\ge 1}\ker(D_{A^n}),
\end{split}
\end{equation}
\[
\begin{array}{l}
\tau \;\mbox{is controllable}\iff \bigcap\limits_{n\ge
1}\ker(D_{A^{*n}})=\{0\}\iff\mbox{the operator}\; A^*\;\mbox{does
not
contain a shift},\\
\tau \;\mbox{is observable}\iff \bigcap\limits_{n\ge
1}\ker(D_{A^{n}})=\{0\}\iff\mbox{the operator}\; A\;\mbox{does not
contain a shift}.
\end{array}
\]
It follows that a conservative system is simple if and only if the
state space operator is completely non-unitary \cite{Br1}.

In \cite{AHS1} we used Theorem \ref{ParContr1} for connections
between the passive system
$$\tau=\left\{\begin{bmatrix}D&C\cr
B&A\end{bmatrix};\sM,\sN,\sH\right\},$$
 its transfer function
$\Theta_\tau(\lambda)$, and the characteristic function of $A$. In
particular, an immediate consequence of \eqref{twee} is the
following relation
\begin{equation}
\label{exptran}
 \Theta_\tau(\lambda)=K\Phi_{A^*}(\lambda)M+D_{K^*}XD_{M},\quad
\lambda\in\dD,
\end{equation}
where $\Phi_{A^*}(\lambda)$ is the characteristic function of $A^*$.

Recall that if $\Theta(\lambda)\in{\bf S}(\sH_1,\sH_2)$ then there
is a uniquely determined decomposition \cite[Proposition V.2.1]{SF}
\[
\Theta(\lambda)=\begin{bmatrix}\Theta_p(\lambda)&0\cr
0&\Theta_u\end{bmatrix}:\begin{array}{l}\sD_{\Theta(0)}\\\oplus\\\ker
D_{\Theta(0)}\end{array}\to \begin{array}{l}
 \sD_{\Theta^*(0)}\\\oplus\\\ker D_{\Theta^*(0)}\end{array},
\]
where $\Theta_p(\lambda)\in{\bf
S}(\sD_{\Theta(0)},\sD_{\Theta^*(0)})$,  $\Theta_p(0)$ is a pure
contraction and $\Theta_u$ is a unitary constant. The function
$\Theta_p(\lambda)$ is called the \textit{pure part} of
$\Theta(\lambda)$ (see {\cite{BC}). If $\Theta(0)$ is isometric
(co-isometric) then the pure part is of the form
$\Theta_p(\lambda)=0\in {\bf S}(\{0\},\sD_{\Theta^*(0)})$
($\Theta_p(\lambda)=0\in {\bf S}(\sD_{\Theta(0)}, \{0\})$).

From \eqref{twee} and \eqref{exptran} we get the following
statement.
\begin{proposition}
\label{PURE} Let
$$\tau=\left\{\begin{bmatrix}D&C\cr
B&A\end{bmatrix};\sM,\sN,\sH\right\}
$$
be a a simple conservative system and let $\Theta(\lambda)$ be its
transfer function. Then
\begin{equation}\label{defnumb}
\begin{array}{l}
\dim\sD_A=\dim\sD_{\Theta^*(0)}=\dim(\sN\ominus\ker
C^*),\\
\dim\sD_{A^*}=\dim\sD_{\Theta(0)}=\dim(\sM\ominus\ker B),
\end{array}
\end{equation}
and the pure part of $\Theta$ coincides with the Sz.-Nagy--Foias
characteristic function of $A^*$.

In addition

1) if $\Theta(0)$ is isometric then $B=0$, $A$ is a co-shift of
multiplicity $\dim\sD_{\Theta^*(0)},$ and the system $\tau$ is
observable;

2) if $\Theta(0)$ is co-isometric then $C=0$, $A$ is a unilateral
shift of multiplicity $\dim\sD_{\Theta(0)}$, and the system $\tau$
is controllable.
\end{proposition}

\begin{proof}
According to Theorem \ref{ParContr1} the operator
\[
T=\begin{bmatrix}D&C\cr B&A\end{bmatrix}:
\begin{array}{l}\sM\\\oplus\\\sH\end{array}\to
\begin{array}{l}\sN\\\oplus\\\sH\end{array}
\]
takes the form \eqref{twee}. Since $T$ is unitary, from
\eqref{contr} we get that the operators $K\in\bL(\sD_A,\sN)$ and
$M^*\in\bL(\sD_{A^*},\sM)$ are isometries and the operator $X\in
\bL(\sD_{M},\sD_{K^*})$ is unitary. From \eqref{exptran} it follows
that the pure part of $\Theta$ is given by
\[
\Theta(\lambda)\uphar\ran{M^*}=K\Phi_{A^*}(\lambda)M\uphar\ran
M^*:\ran M^*\to\ran K.
\] Thus, the pure part
 of $\Theta$ coincides with $\Phi_{A^*}$.
Since $\ran M^*=\sD_{A^*}$, $\ran K^*=\sD_{A}$,
\[
\begin{array}{l}
D=\Theta(0)=K\Phi_{A^*}(0)M^*=-KA^*M^*,\;D^*=\Theta^*(0)=-MAK^*,\\
\ran K=\sN\ominus\ker K^*=\sN\ominus\ker C^*,\\
\ran M^*=\sM\ominus\ker M=\sM\ominus\ker B,
\end{array}
\]
 we get \eqref{defnumb}.

Suppose $D=\Theta(0)$ is an isometry. Then the pure part of $\Theta$
is $0\in{\bf S}(\{0\},\sD_{D^*})$. It follows that $M=B=0$ and
$\sD_{A^*}=\{0\}$. Hence, $A$ is co-isometric and since $A$ is a
completely non-unitary contraction, it is a co-shift of multiplicity
$\dim\sD_{A}=\dim\sD_{\Theta^*(0)}$, and the system $\tau$ is
observable. Similarly the statement 2) holds.
\end{proof}

In this paper we will use a parametrization of a contractive block-
operator matrix based on a fixed upper left block
$D\in\bL(\sM,\sN)$. With this aim we reformulate Theorem
\ref{ParContr1} and Corollary \ref{iscois1}.
\begin{theorem}
\label{ParContr}The operator matrix
\[
T= \begin{bmatrix} D&C\cr B&A \end{bmatrix}:
\begin{array}{l}\sM\\\oplus\\\sH\end{array}\to
\begin{array}{l}\sN\\\oplus\\\sK\end{array}
\]
is a contraction if and only if $D\in\bL(\sM,\sN)$ is a contraction
and the entries $A$,$B$, and $C$ take the form
\begin{equation}
\label{BLOCKS}
\begin{split}
&
B=FD_D,\; C=D_{D^*}G,\\
& A=-FD^*G+D_{F^*}LD_{G},
\end{split}
\end{equation}
where the operators $F\in\bL(\sD_D,\sK)$, $G\in\bL(\sH,\sD_{D^*})$
and $L\in\bL(\sD_{G},\sD_{F^*})$ are contractions. Moreover,
operators $F,\,G,$ and $L$ are uniquely determined. Furthermore,the
following equality holds
\begin{equation}
\label{contr}
\begin{array}{l}
\left\|D_T\begin{bmatrix}h\cr f\end{bmatrix}\right\|^2=
||D_F\left(D_Dh-D^*Gf\right)-F^*LD_{G}f||^2+||D_{L}D_{G}f||^2,\\
\qquad  h\in\sM,\; f\in\sH
\end{array}
\end{equation}
and
\begin{equation}
\label{*contr}
\begin{array}{l}
\left\|D_{T^*}\begin{bmatrix}\f\cr g\end{bmatrix}\right\|^2=
||D_{G^*}\left(D_{D^*}\f-DF^*g\right)-GL^*D_{F^*}g||^2+||D_{L^*}D_{F^*}g||^2,\\
\qquad  \f\in\sN,\;g\in\sK.
\end{array}
\end{equation}
\begin{enumerate}
\item the operator $T$ is isometric if
and only if
\[
D_FD_D=0,\; D_LD_G=0,
\]
\item the operator
$T$ is co-isometric if and only if
\[
D_{G^*}D_{D^*}=0,\; D_{L^*}D_{F^*}=0,
\]
\item if $T$ is unitary then $D_{F^*}=0\iff D_G=0$.
\end{enumerate}
\end{theorem}
Let us give connections between the parametrization of a unitary
block-operator matrix $T$ given by \eqref{twee} and \eqref{BLOCKS}.
\begin{proposition}
\label{uncion} Let
\[
\begin{array}{l}
T=
\begin{bmatrix} -KA^*M+D_{K^*}XD_{M}& KD_{A}\cr
D_{A^*}M
&A\end{bmatrix}=\\
\qquad\quad=\begin{bmatrix}D&D_{D^*}G\cr FD_D
&-FD^*G+D_{F^*}LD_{G}\end{bmatrix}:
\begin{array}{l}\sM\\\oplus\\\sH\end{array}\to
\begin{array}{l}\sN\\\oplus\\\sH\end{array}
\end{array}
\]
be a unitary operator matrix. Then
\[
\sD_D=\ran M^*,\; \sD_{D^*}=\ran K,
\]
\[
F^*=M^*P_{\sD_{A^*}},\; G=KP_{\sD_{A}},
\]
\[
GFf=KP_{\sD_{A}}M f,\; f\in\sD_D,
\]
\[L=A\uphar\ker D_A.
\]
\end{proposition}
\begin{proof} Since $D= -KA^*M+D_{K^*}XD_{M}$, we have
\[
\begin{array}{l}
||D_Df||^2=||D_{A^*}Mf||^2+||(D_KA^*M-K^*XD_M)f||^2+||D_XD_Mf||^2,\; f\in\sM,\\
||D_{D^*}g||^2=||D_{A}K^*g||^2+||(D_{M^*}AK^*-MX^*D_{K^*})g||^2+||D_{X^*}D_{K^*}g||^2,\;g\in\sN.
\end{array}
\]
By Corollary \ref{iscois1} the operators $K$ and $M^*$ are
isometries and $X\in{\bf L}(\sD_{M},\sD_{K^*})$ is unitary operator.
It follows that
\[
||D_Df||^2=||D_{A^*}Mf||^2,\; f\in\sM,\;
||D_{D^*}g||^2=||D_{A}K^*g||^2,\;g\in\sN.
\]
Hence, $D^2_D=M^*D^2_{A^*}M,\; D^2_{D^*}=KD^2_{A}K^*$. Since $K$ and
$M^*$ are isometries, we obtain
\[
D_D=M^*D_{A^*}M,\; D_{D^*}=KD_AK^*.
\]
It follows that $\sD_D=\ran M^*,\; \sD_{D^*}=\ran K,$
$D_{A^*}M=FM^*D_{A^*}M,$ and $D_{A}K^*=G^*KD_{A}K^*.$ Therefore,
\[
FM^*=I_{\sD_{A^*}}, \;G^*K=I_{\sD_{A}}.
\]
It follows
\[
F=M\uphar\sD_D,\; G^*=K^*\uphar\sD_{D^*}.
\]
 Hence, $F^*=M^*P_{\sD_{A^*}}$ and $G=KP_{\sD_A}$.
 In addition
 \[
 D^2_{F^*}=I_\sH-MM^*P_{\sD_{A^*}}=P_{\ker D_{A^*}},\;
 D^2_G=I_\sH-K^*KP_{\sD_{A}}=P_{\ker D_A},
 \]
\[
\begin{array}{l}
-FD^*G=-F(-M^*AK^*+D_MX^*D_{K^*})KP_{\sD_A}=AP_{\sD_A},\\
A=-FD^*G+D_{F^*}LD_G=AP_{\sD_A}+P_{\ker D_{A^*}}LP_{\ker D_A}.
\end{array}
\]
On the other hand
\[
A=AP_{\sD_A}+AP_{\ker D_A}.
\]
Hence $L=A\uphar\ker D_{A}$.

\end{proof}

 Let
$D:\sM\to\sN$ be a contraction with nonzero defect operators and let
 \[
Q=\begin{bmatrix}0&G\cr F&S
\end{bmatrix}:\begin{array}{l}\sD_{D}\\\oplus\\\sH\end{array}\to
\begin{array}{l}\sD_{D^*}\\\oplus\\\sK\end{array}
\]
 be a bounded operator.
 Define the transformation (see\cite{ARL1})
 \begin{equation}
 \label{transform}
 \cM_D(Q)=
\begin{bmatrix}D&0\cr 0&-FD^*G
\end{bmatrix}+
\begin{bmatrix}D_{D^*}&0\cr 0&I_\sK
\end{bmatrix}\begin{bmatrix}0&G\cr F&S
\end{bmatrix}\begin{bmatrix}D_D&0\cr 0&I_\sH
\end{bmatrix}.
\end{equation}
Clearly, the operator $T=\cM_D(Q)$ has the following matrix form
\[
T=\begin{bmatrix}D&D_{D^*}G&\cr FD_D &S-FD^*G
\end{bmatrix}:\begin{array}{l}\sM\\\oplus\\\sH\end{array}\to
\begin{array}{l}\sN\\\oplus\\\sK\end{array}.
\]
\begin{proposition}
\label{trans}\cite{ARL1}. Let $\sH,\sM,\sN$ be separable Hilbert
spaces and let
 $D:\sM\to\sN$ be a contraction with nonzero defect operators. Let
$Q=\begin{bmatrix}0&G\cr F&S
\end{bmatrix}:\begin{array}{l}\sD_{D}\\\oplus\\\sH\end{array}\to
\begin{array}{l}\sD_{D^*}\\\oplus\\\sH\end{array}$ be a bounded operator.
 Then
 \begin{enumerate}
 \item
 \[T=\cM_D(Q)=\begin{bmatrix}D&C\cr
B&A\end{bmatrix}:\begin{array}{l}\sD_{D}\\\oplus\\\sH\end{array}\to
\begin{array}{l}\sD_{D^*}\\\oplus\\\sH\end{array}
\]
is a contraction if and only if $Q$ is a contraction. $T$ is
isometric (co-isometric) if and only if $Q$ is isometric
(co-isometric);
\item holds the relations
\begin{equation}
\label{min11} \left\{
\begin{split}
& \bigcap_{n=0}^\infty\ker \left(B^*A^{*n}\right)=
\bigcap_{n=0}^\infty\ker
\left(F^*S^{*n}\right),\\
&\bigcap_{n=0}^\infty\ker \left(CA^n\right)=\bigcap_{n=0}^\infty\ker
\left(GS^{n}\right).
\end{split}
\right.
\end{equation}
\end{enumerate}
\end{proposition}

\section {The M\"obius representations}
Let $T:\sH_1\to \sH_2$ be a contraction. In \cite{Shtraus} and
\cite{Shmul1} were studied the fractional-linear transformations of
the form
\begin{equation}
\label{unflt} Z\to
Q=T+D_{T^*}Z(I_{\sD_T}+T^*Z)^{-1}D_T=T+D_{T^*}(I_{\sD_{T^*}}+ZT^*)^{-1}ZD_T
\end{equation}
defined on the set $\cV_{T^*}$ of all contractions
$Z\in\bL(\sD_T,\sD_{T^*})$ such that $-1\in\rho(T^*Z).$ The
following result holds.
\begin{theorem} \cite{Shmul1}
\label{SHMUL} Let the $T\in\bL(\sH_1,\sH_2)$ be a contraction and
let $Z\in\cV_{T^*}$. Then $Q=T+D_{T^*}Z(I_{\sD_T}+T^*Z)^{-1}D_T$ is
a contraction,
\begin{equation}
\label{DEFECT11} ||D_{Q}f||^2=||D_Z(I_{\sD_T}+T^*Z)^{-1}D_Tf||,\;
f\in \sH_1,
\end{equation}
$\ran D_Q\subseteq\ran D_T$, and $\ran D_Q=\ran D_T$ if and only if
$||Z||< 1.$ Moreover, if $Q\in\bL(\sH_1,\sH_2)$ is a contraction and
$Q=T+D_{T^*}XD_T,$ where $X\in\bL(\sD_T,\sD_{T^*})$ then
$X\in\cV_{T^*}$,
 $$Z=X(I_{\sD_T}-T^*X)^{-1}\in\cV_{T^*},$$
and the operator $Q$ takes the form
$Q=T+D_{T^*}Z(I_{\sD_T}+T^*Z)^{-1}D_T$.
\end{theorem}
Observe that from \eqref{unflt} one can derive the equalities
\[
\begin{array}{l}
I_{\sH_2}-QT^*=D_{T^*}(I_{\sD_{T^*}}+ZT^*)^{-1}D_{T^*},\\
 Z\uphar\ran
D_T=D_{T^*}(I_{\sH_2}-QT^*)^{-1}(Q-T)D^{-1}_T.
\end{array}
\]
 The
transformation \eqref{unflt} is called in \cite{Shmul1} the unitary
linear-fractional transformation. It is easy to see that if
$||T||<1$ then the closed unit operator ball in $\bL(\sH_1, \sH_2)$
belongs to the set $\cV_{T^*}$ and, moreover
\[
\begin{split}
&T+D_{T^*}Z(I_{\sD_T}+T^*Z)^{-1}D_T=D^{-1}_{T^*}(Z+T)(I_{\sD_T}+T^*Z)^{-1}D_{T}=\\
&\qquad\qquad\qquad=D_{T^*}(I_{\sD_{T^*}}+ZT^*)^{-1}(Z+T)D^{-1}_{T}
\end{split}
\]
for all $Z\in\bL(\sH_1,\sH_2),$ $||Z||\le 1.$ Thus, the
transformation \eqref{unflt} is an operator analog of a well known
M\"obius transformation of the complex plane
\[
z\to \frac{z+t}{1+\bar t z},\;|t|\le 1.
\]
The next theorem is a version of a more general result established
by Yu.L.~Shmul'yan in \cite{Shmul2}.
\begin{theorem} \cite{Shmul2}
\label{SHMUL1} Let $\sM$ and $\sN$ be Hilbert spaces and let the
function $\Theta(\lambda)$ be from the Schur class ${\bf
S}(\sM,\sN).$ Then
\begin{enumerate}
\item the linear manifolds
$\ran D_{\Theta(\lambda)}$ and $\ran D_{\Theta^*(\lambda)}$ do not
depend on $\lambda\in\dD,$
\item for arbitrary $\lambda_1,$ $\lambda_2$, $\lambda_3$ in $\dD$ the function $\Theta(\lambda)$ admits the representation
\[
\Theta(\lambda)=\Theta(\lambda_1)+D_{\Theta^*(\lambda_2)}\Psi(\lambda)D_{\Theta(\lambda_3)},
\]
where $\Psi(\lambda)$ is a holomorphic in $\dD$ and
$\bL\left(\sD_{\Theta(\lambda_3)},\sD_{\Theta^*(\lambda_2)}\right)$-valued
function.
\end{enumerate}
\end{theorem}
Now using Theorems \ref{SHMUL} and \ref{SHMUL1} we get Theorem
\ref{MO}. Recall that the representation \eqref{MREP} of a function
$\Theta(\lambda)\in{\bf S}(\sM,\sN)$ is called the M\"obius
representation of $\Theta$ and the function $Z(\lambda)\in {\bf
S}(\sD_{\Theta(0)},\sD_{\Theta^*(0)})$ is called the M\"obius
parameter of $\Theta$.

The next result established in \cite{ARL1} provides connections
between the realizations of $\Theta(\lambda)$ and $Z(\lambda)$ as
transfer functions of passive systems.
\begin{theorem}
\label{TT1}\cite{ARL1}. \begin{enumerate}
\item Let $\tau=\left\{\begin{bmatrix}D&C\cr B&A
\end{bmatrix};\sM,\sN,\sH\right\}$ be a passive system and let
\[
T=\begin{bmatrix}D&C\cr B&A
\end{bmatrix}
=\begin{bmatrix}D& D_{D^*}G\cr FD_D &-FD^*G+D_{F^*}LD_{G}
\end{bmatrix}:\begin{array}{l}\sM\\\oplus\\\sH\end{array}\to
\begin{array}{l}\sN\\\oplus\\\sH\end{array}.
\]
Let $\Theta(\lambda)$ be the transfer function of $\tau$. 
 Then
 \begin{enumerate}
\def\labelenumi{\rm (\roman{enumi})}
 \item the M\"obius parameter $Z(\lambda)$ of the function $\Theta(\lambda)$ is the
transfer function of the passive system
\[
\nu=\left\{\begin{bmatrix}0& G\cr
F&D_{F^*}LD_G\end{bmatrix};\sD_{D},\sD_{D^*},\sH\right\};
\]
\item
the system $\tau$ isometric (co-isometric) $\Rightarrow$ the system
$\nu$ isometric (co-isometric);
\item  the equalities $\sH^c_\nu =\sH^c_\tau$, $\sH^o_\nu =\sH^o_\tau$ hold and hence the system $\tau$ is
controllable (observable) $\Rightarrow$ the system $\nu$ is
controllable (observable),
 the system $\tau$ is simple (minimal)
$\Rightarrow$ the system $\nu$ is simple (minimal).
\end{enumerate}
\item
 Let $\Theta(\lambda)\in {\bf S}(\sM,\sN)$ and let $Z(\lambda)$ be the M\"obius parameter of
 $\Theta(\lambda)$. Suppose that the transfer function
of the linear system
\[ \nu'=\left\{\begin{bmatrix}0&G\cr F
&S\end{bmatrix};\sD_{\Theta(0)},\sD_{\Theta^*(0)}, \sH\right\}
\]
coincides with $Z(\lambda)$ in a neighborhood of the origin. Then
the transfer function of the linear system
\[
\tau'=\left\{\begin{bmatrix}\Theta(0)&D_{\Theta^*(0)}G\cr
FD_{\Theta(0)} &-F\Theta^*(0)G+S\end{bmatrix};\sM,\sN,\sH\right\}
\]
coincides with
 $\Theta(\lambda)$ in a neighborhood of the origin. Moreover
 \begin{enumerate}
 \item
 the equalities $\sH^c_{\tau'} =\sH^c_{\nu'}$, $\sH^o_{\tau'} =\sH^o_{\nu'}$ hold, and hence the system
$\nu'$ is controllable (observable) $\Rightarrow$ the system $\tau'$
is controllable (observable), the system $\nu'$ is simple
$\Rightarrow$ the system $\tau'$ is simple (minimal),
\item the system $\nu'$ is passive $\Rightarrow$ the system $\tau'$
is passive (minimal),
\item
the system $\nu'$ isometric (co-isometric) $\Rightarrow$ the system
$\tau'$ isometric (co-isometric).
\end{enumerate}
\end{enumerate}
\end{theorem}
\begin{corollary}
\label{zero} 1) The equivalences
\[
\begin{array}{l}
\f_\Theta(\lambda)=0 \iff \f_Z(\lambda)=0,\\
\psi_\Theta(\lambda)=0\iff \psi_Z(\lambda)=0
\end{array}
\]
hold.

2) Let $||\Theta(0)\uphar\sD_{\Theta(0)}||<1.$ Suppose
$\f(\lambda)\in{\bf S}(\sM,\sL)$ ($\psi(\lambda)\in{\bf
S}(\sK,\sN)$) and
\[
\begin{array}{l}
\f^*(\xi)\f(\xi)=D^2_{\Theta(\xi)}\quad\mbox{for almost
all}\quad\xi\in\dT\\
\left(\psi(\xi)\psi^*(\xi)=D^2_{\Theta^*(\xi)}\quad\mbox{for almost
all}\quad\xi\in\dT \right).
\end{array}
\]
Then
\[
\begin{array}{l}
\wt\f(\lambda):=\f(\lambda)D^{-1}_{\Theta(0)}(I_{\sD_{\Theta(0)}}+\Theta^*(0)Z(\lambda))
\in{\bf S}(\sD_{\Theta(0)},\sL)\\
\left(
\wt\psi(\lambda):=(I_{\sD_{\Theta^*(0)}}+Z(\lambda)\Theta^*(0))
D^{-1}_{\Theta^*(0)}P_{\sD_{\Theta^*(0)}}\psi(\lambda) \in{\bf
S}(\sK,\sD_{\Theta^*(0)}) \right)
\end{array}
\]
and
\[
\begin{array}{l}
\wt \f^*(\xi)\wt\f(\xi)=D^2_{Z(\xi)}\quad\mbox{for almost
all}\quad\xi\in\dT\\
\left( \wt \psi(\xi)\wt\psi^*(\xi)=D^2_{Z^*(\xi)}\quad\mbox{for
almost all}\quad\xi\in\dT \right).
\end{array}
\]
In particular,
\[
\Theta(\lambda)\quad\mbox{is inner (co-inner)}\;\iff
Z(\lambda)\quad\mbox{is inner (co-inner)}.
\]
\end{corollary}
\begin{proof}
1) Let $\f_\Theta(\lambda)=0$ ($\psi_\Theta(\lambda)=0$) and let
$\tau=\left\{\begin{bmatrix}D&C\cr B&A
\end{bmatrix};\sM,\sN,\sH\right\}$ be a simple conservative system with transfer function
$\Theta(\lambda)$. By Theorem \ref{DBR} the system $\tau$ is
observable (controllable). As it is proved above the corresponding
system $\nu$ with transfer function $Z(\lambda)$ is conservative and
observable (controllable). Theorem \ref{DBR} yields that
$\f_Z(\lambda)=0$ ($\psi_Z(\lambda)=0$).

Conversely. Let $\f_Z(\lambda)=0$ ($\psi_Z(\lambda)=0$) and let
$\nu'$ be a simple conservative system with transfer function
$Z(\lambda)$. Again by Theorem \ref{DBR} the system $\nu'$ is
observable (controllable). As it is already proved the corresponding
system $\tau'$ with transfer function $\Theta(\lambda)$ is
conservative and observable (controllable) as well. Now Theorem
\ref{DBR} yields that $\f_\Theta(\lambda)=0$
($\psi_\Theta(\lambda)=0$).

2) Let $||\Theta(0)\uphar\sD_{\Theta(0)}||<1$. Since
\[
\Theta^*(0)\uphar\sD_{\Theta^*(0)}=\left(\Theta(0)\uphar\sD_{\Theta(0)}\right)^*,
\]
we get $||\Theta^*(0)\uphar\sD_{\Theta^*(0)}||<1$. It follows that
the operators $D_{\Theta(0)}\uphar\sD_{\Theta(0)}$ and
$D_{\Theta^*(0)}\uphar\sD_{\Theta^*(0)}$ have bounded inverses. From
\eqref{DEFECT11} we obtain the relation
\[
||D_{\Theta(\lambda)}D^{-1}_{\Theta(0)}(I_{\sD_{\Theta(0)}}+\Theta^*(0)Z(\lambda))f||^2=
||D_{Z(\lambda)}f||^2,\;\lambda\in\dD,\;f\in\sD_{\Theta(0)}.
\]
The non-tangential limits $\Theta(\xi)$ and $Z(\xi)$ exist for
almost all $\xi\in\dT$. It follows the relation
\[
||D_{\Theta(\xi)}D^{-1}_{\Theta(0)}(I_{\sD_{\Theta(0)}}+\Theta^*(0)Z(\xi))f||^2=
||D_{Z(\xi)}f||^2,\;f\in\sD_{\Theta(0)}.
\]
for almost all $\xi\in\dT.$ This completes the proof.
 \end{proof}
\begin{theorem}
\label{EXX} Let $A$ be a completely non-unitary contraction in the
Hilbert space $\sH$ and let $Z(\lambda)$ be the  M\"obius parameter
of the Sz.Nagy--Foias characteristic function of $A$. Then
$Z(\lambda)$ is the characteristic function of the operator
$\cA_{1,0}=AP_{\ker D_A}$(see \eqref{Anm} and \eqref{CA}). Moreover,
the following statements are equivalent
\begin{enumerate}
\def\labelenumi{\rm (\roman{enumi})}
\item the unitary equivalent operators $A_{1,0}$ and $A_{0,1}$ are  unilateral shifts (co-shifts),
\item $\sD_A\subset\sD_{A^*}$ ($\sD_{A^*}\subset\sD_A$),
\item
the M\"obius parameter takes the form $Z(\lambda)=\lambda I_{\sD_A}$
($Z^*(\bar\lambda)=\lambda I_{\sD_{A^*}}).$
\end{enumerate}
\end{theorem}
\begin{proof} The system
$$\Sigma=\left\{\begin{bmatrix}-A& D_{A^*} \cr D_{A}& A^*
\end{bmatrix};  \sD_{A},\sD_{A^*},\sH\right\}$$
 is conservative and simple and its transfer function
\[
\Phi(\lambda)=\left(-A+\lambda D_{A^*}(I_\sH-\lambda
A^*)^{-1}D_{A}\right)\uphar\sD_{A}
\]
is the characteristic function of $A$. Let
 $F$ and $G^*$ be the embedding  of the subspaces $\sD_{A}$ and
$\sD_{A^*}$ into $\sH$, respectively. It follows that
\[
D_{F^*}=P_{\ker D_{A}},\;D_{G}=P_{\ker D_{A^*}}.
\]
Let $L=A^*\uphar\ker D_{A^*}$. Then
\[
A^*=A^*P_{\sD_{A^*}}+A^*P_{\ker D_{A^*}}=-F(-A^*)G+D_{F^*}LD_G
\]
 Let
\[
\Phi(\lambda)=\Phi(0)+D_{\Phi^*(0)}Z(\lambda)(I+\Phi^*(0)Z(\lambda))^{-1}D_{\Phi(0)},\;\lambda\in\dD
\]
be the M\"obius representation of the function $\Phi(\lambda)$. By
Theorem \ref{TT1} the system
\[
\nu=\left\{\begin{bmatrix}0& P_{\sD_{A^*}}\cr I_{\sD_A}&A^*P_{\ker
D_{A^*}}\end{bmatrix};\sD_{A},\sD_{A^*},\sH \right\}
\]
is conservative and simple and its transfer function is the function
$Z(\lambda)$, i.e.,
\[
Z(\lambda)=\lambda P_{\sD_{A^*}}\left(I_\sH-\lambda A^*P_{\ker
D_{A^*}}\right)^{-1}\uphar\sD_{A},\; |\lambda|<1.
\]
This function is exactly the Sz.-Nagy--Foias characteristic function
of the partial isometry $\cA_{1,0}=AP_{\ker\sD_{A}}$.

Suppose $A_{1,0}=P_{\ker D_{A}}A\uphar\ker D_A$ is a unilateral
shift. Since $A\ker D_A=\ker D_{A^*}$, we have $\ker
D_{A^*}\subset\ker D_A$. Equivalently $\sD_{A}\subset\sD_{A^*}$.
Hence,
\[
P_{\ker D_{A^*}}\uphar\sD_A=0 \quad\mbox{and}\quad (A^* P_{\ker
D_{A^*}})^n\uphar\sD_{A}=0\quad\mbox{for all}\quad n\in\dN.
\]
Therefore,
\[
Z(\lambda)=\lambda P_{\sD_{A^*}}\uphar \sD_{A}=\lambda I_{\sD_A}.
\]
Conversely, suppose $Z(\lambda)=\lambda I_{\sD_A}$. Then
$\sD_{A}\subset \sD_{A^*}$ $\Rightarrow$ $\ker D_{A}\supset\ker
D_{A^*}$. It follows
\[
A\ker D_{A}\subset \ker D_{A}\Rightarrow A_{1,0}\quad\mbox{is
isometry}.
\]
Since the operator $A_{1,0}$ is completely non-unitary, it is a
unilateral shift.
\end{proof}
\begin{corollary}
\label{COO} Let $A$ be a completely non-unitary contraction in a
separable Hilbert space $\sH$ and let $||A\uphar\sD_{A}||<1(\iff \ran D_A=\cran D_A)$. Then
the following statements are equivalent
\begin{enumerate}
\def\labelenumi{\rm (\roman{enumi})}
\item $A\in C_{\cdot\,0}$ (respect.,$ A\in C_{0\,\cdot})$,
\item $\cA_{1,0}\in C_{\cdot\,0}$ (respect., $ \cA_{1,0}\in C_{0\,\cdot})$.
\end{enumerate}
\end{corollary}
\begin{proof} By \eqref{Char.funct.} we have
$\Phi_A(0)=-A\uphar\sD_A.$  Then in accordance with \cite{SF},
Corollary \ref{zero}, and Theorem \ref{EXX} we get the equivalences
\[
\begin{array}{l}
A\in C_{\cdot\,0}\,\,(C_{0\,\cdot})\iff \Phi_A(\lambda)\quad\mbox{is
inner (co-inner)}\quad\iff Z(\lambda)\quad\mbox{is inner
(co-inner)}\\
\iff \cA_{1,0}\in C_{\cdot\,0}\,\,(C_{0\,\cdot}).
\end{array}
\]

\end{proof}
\section{Realizations of the Schur iterates}
\subsection{Realizations of the first Schur iterate}
\begin{proposition}
\label{real} Let $\sH$, $\sL$, $\sK$ be Hilbert spaces and let $F\in
\bL(\sL,\sH)$, $G\in \bL(\sH,\sK)$  and $L\in\bL(\sD_G,\sD_{F^*})$
be contractions. Let $Z_\nu(\lambda)$ be the transfer function of
the system
\begin{equation}
\label{NU} \nu=\left\{\begin{bmatrix}0&G\cr F&
D_{F^*}LD_G\end{bmatrix};\sL,\sK, \sH\right\}
\end{equation}
Then the function $\Gamma(\lambda)=\lambda^{-1}Z_\nu(\lambda)$ is
the transfer function of the passive systems
\[
\eta_1=\left\{\begin{bmatrix}GF &GD_{F^*}\cr LD_GF
&LD_GD_{F^*}\end{bmatrix};\sL,\sK,\sH\right\},\;
\eta_2=\left\{\begin{bmatrix}GF&GD_{F^*}\wt L\cr  D_G F&D_G
D_{F^*}\wt L\end{bmatrix};\sL,\sK, \sH\right\},
\]
where $\wt L=L P_{\sD_G}$.

Suppose that the subspaces $\sH_{\zeta_1}=D_{F^*}$  and
$\sH_{\zeta_2}=\sD_G$ are nontrivial. Then the transfer functions of
the passive systems
\begin{equation}
\label{ZETA12} \zeta_1=\left\{\begin{bmatrix}GF &GD_{F^*}\cr LD_GF
&LD_GD_{F^*}\end{bmatrix};\sL,\sK,\sH_{\zeta_1}\right\},\;
\zeta_2=\left\{\begin{bmatrix}GF&GD_{F^*}\wt L\cr  D_G F&D_G
D_{F^*}\wt L\end{bmatrix};\sL,\sK,\sH_{\zeta_2}\right\}
\end{equation}
are equal to $\Gamma(\lambda).$ Moreover, for the orthogonal
complements to the controllable and observable subspaces of the
systems $\nu$, $\zeta_1$, and $\zeta_2$ hold the following relations
\begin{equation}
\label{RELAT}
\begin{array}{l}
\left(\sH^{c}_\nu\right)^\perp=\left(\sH^{c}_{\zeta_1}\right)^\perp\cap
\ker
F^*,\;\left(\sH^{o}_\nu\right)^\perp=\left(\sH^{o}_{\zeta_2}\right)^\perp\cap
\ker G,\\
D_G\left(\sH^{c}_{\zeta_2}\right)^\perp\subset
\left(\sH^{c}_\nu\right)^\perp,\;D_{F^*}\left(\sH^{o}_{\zeta_1}\right)^\perp\subset
\left(\sH^{o}_{\nu}\right)^\perp.
\end{array}
\end{equation}
If the operators $G^*$ and $F$ are isometries, then
\begin{equation}
\label{RELAT1}
\left(\sH^{o}_{\zeta_1}\right)^\perp=\left(\sH^{o}_{\nu}\right)^\perp\cap
\ker{F^*},\; \left(\sH^{c}_{\zeta_2}\right)^\perp=
\left(\sH^{c}_\nu\right)^\perp\cap\ker G.
\end{equation}
\end{proposition}
\begin{proof} We have
\[
Z_\nu(\lambda)=\lambda G(I_\sH-\lambda D_{F^*}LD_G)^{-1}F.
\]
Hence
\[
\Gamma(\lambda)=\frac{Z_\nu(\lambda)}{\lambda}=G(I_\sH-\lambda
D_{F^*}LD_G)^{-1}F
\]
and $\Gamma(0)=GF$. It follows that
\[
\begin{array}{l}
\Gamma(\lambda)-\Gamma(0)=G(I_\sH-\lambda
D_{F^*}LD_G)^{-1}F-GF=\lambda GD_{F^*}LD_G(I_\sH-\lambda
D_{F^*}LD_G)^{-1}F\\
=\lambda GD_{F^*}(I_\sH-\lambda LD_GD_{F^*})^{-1}LD_GF= \lambda GD_{F^*}(I_\sH-\lambda \wt LD_GD_{F^*})^{-1}\wt LD_GF\\=
\lambda
GD_{F^*}\wt L(I_\sH-\lambda D_G D_{F^*}\wt L)^{-1}D_G F,
\end{array}
\]
\begin{equation}
\label{GAMMN}
\begin{array}{l}
 \Gamma(\lambda)=GF+\lambda GD_{F^*}(I_\sH-\lambda
LD_GD_{F^*})^{-1}LD_GF\\
\qquad\qquad=GF+\lambda GD_{F^*}\wt L(I_\sH-\lambda D_G D_{F^*}\wt
L)^{-1}D_G F.
\end{array}
\end{equation}
The operators
\[
K_1=\begin{bmatrix}GF &GD_{F^*}\cr LD_GF
&LD_GD_{F^*}\end{bmatrix}:\begin{array}{l}\sL\\\oplus\\\sH\\\end{array}\to
\begin{array}{l}\sK\\\oplus\\\sH\\\end{array}
\]
and
\[
K_2=\begin{bmatrix}GF&GD_{F^*}\wt L\cr  D_G F&D_G D_{F^*}\wt
L\end{bmatrix}:\begin{array}{l}\sL\\\oplus\\\sH\end{array}\to
\begin{array}{l}\sK\\\oplus\\\sH\end{array}
\]
are contraction. Actually, let $f\in\sH$ and $h\in\sL$ then one can
check that
\[
\left\|\begin{bmatrix} f\cr
h\end{bmatrix}\right\|^2-\left\|K_1\begin{bmatrix} f\cr
h\end{bmatrix}\right\|^2=||F^*f-D_Fh||^2_\sL+||D_{
L}D_G(D_{F^*}f+Fh)||^2_\sH\ge 0,
\]
\[
\left\|\begin{bmatrix} f\cr
h\end{bmatrix}\right\|^2-\left\|K_2\begin{bmatrix} f\cr
h\end{bmatrix}\right\|^2=||F^*\wt L f-D_{F}h||^2_\sL+||D_{\wt
L}f||^2_\sH\ge 0.
\]
Thus, the systems $\eta_1$, $\eta_2,$ $\zeta_1$, and $\zeta_2$ are
passive and their transfer functions are precisely
$\Gamma(\lambda)$.

Since $\wt L^*\uphar\ker D_{F^*}=0$ and $F^*f=0\iff D_{F^*}f=f$,
$Gh=0\iff D_Gh=h$, by induction one can derive the following
equalities
\begin{equation}
\label{corel} \left\{
\begin{array}{l}
\bigcap\limits_{n\ge 0}\ker\left(F^*(D_{G}L^*D_{F^*})^n\right)=\bigcap\limits_{n\ge 0}\ker\left(F^*(D_G \wt L^*)^{n})\right),\\
\bigcap\limits_{n\ge 0}\ker\left(G(D_{F^*}LD_{G})^n\right)=\bigcap\limits_{n\ge 0}\ker\left(G(D_{F^*}\wt L)^{n}\right),\\
\bigcap\limits_{n\ge 0}\ker\left(F^*D_{G}\wt L^*(D_{F^*}D_{G}\wt
L^*)^n\right)=\bigcap\limits_{n\ge
1}\ker\left(F^*(D_G \wt L^*)^{n})\right),\\
\bigcap\limits_{n\ge 0}\ker\left(GD_{F^*}(\wt
LD_{G}D_{F^*})^n\right)=\bigcap\limits_{n\ge
0}\ker\left(G(D_{F^*}\wt L)^{n}\right),\\
\bigcap\limits_{n\ge 0}\ker\left(F^*D_G(\wt
L^*D_{F^*}D_G)^n\right)=\bigcap\limits_{n\ge 0}\ker\left(F^*(D_G\wt
L^*)^nD_G\right),\\
\bigcap\limits_{n\ge 0}\ker\left(GD_{F^*}\wt L(D_GD_{F^*}\wt
L)^n\right)=\bigcap\limits_{n\ge 1}\ker \left(G(D_{F^*}\wt
L)^n\right).
\end{array}
\right.
\end{equation}
From \eqref{corel} follow the relations \eqref{RELAT} and
\eqref{RELAT1}.
\end{proof}
\begin{theorem}
\label{Schuriso} Let the system
\[
\tau=\left\{\begin{bmatrix}D&D_{D^*}G\cr FD_D&-FD^*G+D_{F^*}LD_G
\end{bmatrix};\sM,\sN,\sH\right\}
\]
be conservative and simple and let $\Theta(\lambda)$ be its transfer
function. Suppose that the first Schur iterate $\Theta_1(\lambda)$
of $\Theta$ is non-unitary constant. Then the systems
\begin{equation}
\label{newzeta}
\begin{array}{l}
\zeta_1=\left\{\begin{bmatrix} GF& G\cr LD_GF&
LD_G\end{bmatrix};\sD_D,\sD_{D^*},\sD_{F^*}\right\},\\
\zeta_2=\left\{\begin{bmatrix} GF& GL\cr D_GF&
D_GL\end{bmatrix};\sD_D,\sD_{D^*},\sD_{G}\right\}
\end{array}
\end{equation}
are conservative and simple and their transfer functions are equal
to $\Theta_1(\lambda)$.
\end{theorem}
\begin{proof}
Because the system $\nu$ is conservative, the operators $F$ and
$G^*$ are isometries. Since $\Theta_1(\lambda)$ is non-unitary
constant, from \eqref{GAMMN} it follows that the operator $GF$ is
non-unitary. Hence by Theorem \ref{ParContr} the subspaces
$\sD_{F^*}$ and $\sD_G$ are nontrivial, and the operator
$L\in\bL(\sD_G,\sD_{F^*})$ is unitary. In addition, $\ker
F^*=\sD_{F^*}$, $\ker G=\sD_G$, and the operators $D_{F^*}$ and
$D_{G}$ are orthogonal projections in $\sH$ onto $\ker F^*$ and
$\ker G$, respectively. One can directly check that the operators
\[
\begin{bmatrix} GF& G\cr LD_GF&
LD_G\end{bmatrix}:\begin{array}{l}\sD_{D}\\\oplus\\\sD_{F^*}\end{array}\to
\begin{array}{l}\sD_{D^*}\\\oplus\\\sD_{F^*}\end{array},\;
\begin{bmatrix} GF& GL\cr D_GF&
D_GL\end{bmatrix}:\begin{array}{l}\sD_D\\\oplus\\\sD_G\end{array}\to
\begin{array}{l}\sD_{D^*}\\\oplus\\\sD_G\end{array}
\]
are unitary. Hence, the systems $\zeta_1$ and $\zeta_2$ are
conservative. Relation \eqref{RELAT} yields in our case that
\[
\left(\sH^{c}_\nu\right)^\perp=\left(\sH^{c}_{\zeta_1}\right)^\perp,\;
\left(\sH^{o}_\nu\right)^\perp=\left(\sH^{o}_{\zeta_2}\right)^\perp.
\]
Taking into account \eqref{RELAT1} and the simplicity of $\nu$ we
get that the systems $\zeta_1$ and $\zeta_2$ are simple.
\end{proof}
\begin{theorem}
\label{IITER} Let $\Theta(\lambda)\in{\bf S}(\sM,\sN)$,
$\Gamma_0=\Theta(0)$ and let $\Theta_1(\lambda)$ be the first Schur
iterate of $\Theta$. Suppose
\[
\tau=\left\{\begin{bmatrix}\Gamma_0&C\cr
B&A\end{bmatrix};\sM,\sN,\sH\right\}
\]
is a simple conservative system with transfer function $\Theta$.
Then the simple conservative system
\[
\nu=\left\{\begin{bmatrix}0&D^{-1}_{\Gamma^*_0}C\cr
D^{-1}_{A^*}B&AP_{\ker
D_{A}}\end{bmatrix},\sD_{\Gamma_0},\sD_{\Gamma^*_0},\sH \right\}
\]
has the transfer function $\lambda \Theta_1(\lambda)$ while the
simple conservative systems
\begin{equation}\label{RealFirst}
\begin{array}{l}
\zeta_{1}=\left\{\begin{bmatrix}D^{-1}_{\Gamma^*_0}C(D^{-1}_{\Gamma_0}B^*)^*&D^{-1}_{\Gamma^*_0}C\uphar\ker
D_{A^*}\cr AP_{\ker D_A}D^{-1}_{A^*}B&P_{\ker D_{A^*}}A\uphar\ker
D_{A^*}\end{bmatrix};\sD_{\Gamma_0},\sD_{\Gamma^*_0},\ker
D_{A^*}\right\},\\
\zeta_{2}=\left\{\begin{bmatrix}D^{-1}_{\Gamma^*_0}C(D^{-1}_{\Gamma_0}B^*)^*&
D^{-1}_{\Gamma^*_0}CA\uphar\ker {D_A}\cr P_{\ker
D_A}D^{-1}_{A^*}B&P_{\ker D_A}A\uphar\ker
D_A\end{bmatrix};\sD_{\Gamma_0},\sD_{\Gamma^*_0},\ker D_A\right\}
 \end{array}
\end{equation}
have transfer functions $\Theta_1(\lambda)$. Here the operators
$D^{-1}_{\Gamma_0},$ $D^{-1}_{\Gamma^*_0},$ and $D^{-1}_{A^*}$ are
the Moore--Penrose pseudo-inverses.
\end{theorem}
\begin{proof}
Let
\[
\begin{array}{l}
T=\begin{bmatrix}\Gamma_0&C\cr
B&A\end{bmatrix}=\begin{bmatrix}\Gamma_0&D_{\Gamma^*_0}G\cr
FD_{\Gamma_0} &-F\Gamma^*_0G+D_{F^*}LD_{G}\end{bmatrix}=\\
\qquad\quad=\begin{bmatrix} -KA^*M+D_{K^*}XD_{M}& KD_{A}\cr D_{A^*}M
&A\end{bmatrix}:
\begin{array}{l}\sM\\\oplus\\\sH\end{array}\to
\begin{array}{l}\sN\\\oplus\\\sH\end{array}.
\end{array}
\]
Then $G=D^{-1}_{\Gamma^*_0}C$, $F^*=D^{-1}_{\Gamma_0}B^*,$
$F=M\uphar\sD_{\Gamma_0},$ $M=D^{-1}_{A^*}B$. According to
Proposition \ref{uncion} we have
\[
D_{F^*}=P_{\ker D_{A^*}}, \;D_G=P_{\ker D_A},\; L=A\uphar\ker D_A.
\]
Hence
\[
\begin{array}{l}
GF=D^{-1}_{\Gamma^*_0}C(D^{-1}_{\Gamma_0}B^*)^*,\;
D_GD_{F^*}L=P_{\ker D_{A}}A\uphar\ker D_{A},\\
D_GF=P_{\ker D_A}M=P_{\ker D_A}D^{-1}_{A^*}B,\;
GD_{F^*}L=D^{-1}_{\Gamma^*_0}CP_{\sD_{A}}A\uphar\ker D_{A},\\
LD_G\uphar\ker D_{A^*}=AP_{\ker D_A}\uphar\ker D_{A^*},\;
LD_GF=AP_{\ker D_A}D^{-1}_{A^*}B.
\end{array}
\]
Note that if $f\in\ker D_{A^*}$ then
\[
AP_{\ker D_A}f=P_{\ker D_{A^*}}AP_{\ker D_A}f=P_{\ker
D_{A^*}}Af-P_{\ker D_{A^*}}AP_{\sD_A}f=P_{\ker D_{A^*}}Af.
\]
Now the statement of theorem follows from Theorem \ref{TT1} and
Theorem \ref{Schuriso}.
\end{proof}

\begin{remark}
\label{RREM1} Since $F^*=D^{-1}_{\Gamma_0}B^*$, we get
$F=\left(D^{-1}_{\Gamma_0}B^* \right)^*\in\bL(\sD_{\Gamma_0},\sH).$
Hence
\[
D^{-1}_{A^*}B\uphar\sD_{\Gamma_0}=\left(D^{-1}_{\Gamma_0}B^*
\right)^*.
\]
 Using the Hilbert spaces and operators defined by \eqref{hnm}
and \eqref{Anm}, we get
\[
P_{\ker D_A}D^{-1}_{A^*}B\uphar\sD_{\Gamma_0}=P_{1,0}
D^{-1}_{A^*}B\uphar\sD_{\Gamma_0}=\left(D^{-1}_{\Gamma_0}\left(B^*\uphar\sH_{1,0}\right)
\right)^*\in\bL(\sD_{\Gamma_0},\sH_{1,0}).
\]
In addition
\[
D^{-1}_{\Gamma^*_0}C(D^{-1}_{\Gamma_0}B^*)^*=\Gamma_1\in\bL(\sD_{\Gamma_0},\sD_{\Gamma^*_0}).
\]
 So,
\begin{equation}
\label{n1}
\begin{array}{l}
\zeta_{1}=\left\{\begin{bmatrix}\Gamma_1& D^{-1}_{\Gamma^*_0}C\cr
A\left(D^{-1}_{\Gamma_0}\left(B^*\uphar\sH_{1,0}\right)
\right)^*&A_{0,1}\end{bmatrix};\sD_{\Gamma_0},\sD_{\Gamma^*_0},\sH_{0,1}\right\},\\
\zeta_{2}=\left\{\begin{bmatrix}\Gamma_1& D^{-1}_{\Gamma^*_0}CA\cr
\left(D^{-1}_{\Gamma_0}\left(B^*\uphar\sH_{1,0}\right)
\right)^*&A_{1,0}\end{bmatrix};\sD_{\Gamma_0},\sD_{\Gamma^*_0},\sH_{1,0}\right\}.
 \end{array}
\end{equation}
It follows that
\[
\begin{array}{l}
\ran\left(D^{-1}_{\Gamma^*_0}C\uphar\sH_{1,0}\right)\subset\ran
D_{\Gamma^*_1},\\
\ran\left(D^{-1}_{\Gamma_0}B^*\uphar\sH_{1,0} \right)\subset\ran
D_{\Gamma_1}
\end{array}
\]
\end{remark}
\subsection{Schur iterates of the characteristic function}
\begin{theorem}
\label{ITERATES} Let $A$ be a completely non-unitary contraction in
a separable Hilbert space $\sH$. Assume $\ker D_A\ne\{0\}$ and let
the contractions $A_{n,m}$ be defined by \eqref{hnm} and
\eqref{Anm}. Then the characteristic functions of the operators
\[
A_{n,0},A_{n-1, 1},\ldots, A_{n-m, m},\ldots A_{1,n-1}, A_{0, n}
\]
coincide with the pure part of the $n$-th Schur iterate of the
characteristic function $\Phi(\lambda)$ of $A$. Moreover, each
operator from the set $\{A_{n-k,k}\}_{k=0}^n$ is
\begin{enumerate}
\item  a unilateral shift (co-shift) if and only if the $n$-th Schur
parameter $\Gamma_{n}$ of $\Phi$ is isometric (co-isometric),
\item  the orthogonal sum of  a unilateral shift and
co-shift if and only if
\begin{equation}
\label{SHCOSH} \sD_{\Gamma_{n-1}}\ne \{0\},\;
\sD_{\Gamma^*_{n-1}}\ne \{0\}\quad\mbox{and}\quad
\Gamma_m=0\quad\mbox{for all}\quad m\ge n.
\end{equation}
\end{enumerate}
 Each subspace from the set $\{\sH_{n-k,k}\}_{k=0}^n$ is trivial
  if and only if $\Gamma_n$ is unitary.
\end{theorem}
\begin{proof} We will prove by induction.
The system
$$\Sigma=\left\{\begin{bmatrix}-A&D_{A^*} \cr  D_{A}&
A^*\end{bmatrix}; \sD_{A},\sD_{A^*}, \sH\right\}$$
 is conservative and simple and its transfer function
$\Phi(\lambda)$ is Sz.-Nagy--Foias characteristic function of $A$.
As in Theorem \ref{EXX}, let
 $F$ and $G^*$ be the embedding  of the subspaces $\sD_{A}$ and
$\sD_{A^*}$ into $\sH$, respectively. Then $D_{F^*}=P_{\ker
D_{A}}=P_{1,0},$ $D_{G}=P_{\ker D_{A^*}}=P_{0,1}$, and
$L=A^*\uphar\ker D_{A^*}\in\bL(\sD_{A^*},\sD_A)$ is unitary
operator. The system
\[
\nu=\left\{\begin{bmatrix}0& P_{\sD_{A^*}}\cr  I_{\sD_A}&A^*P_{\ker
D_{A^*}}\end{bmatrix};\sD_{A},\sD_{A^*}, \sH \right\}
\]
is conservative and simple and its transfer function $Z(\lambda)$ is
the M\"obius  parameter of $\Phi(\lambda)$. Constructing the systems
given by \eqref{newzeta} in Theorem \ref{Schuriso} we get
\[
\zeta_1=\left\{\begin{bmatrix}P_{\sD_{A^*}}\uphar\sD_A&P_{\sD_{A^*}}\uphar\ker
D_A \cr
   A^*P_{\ker D_{A^*}}\uphar\sD_{A}&
A^*P_{\ker D_{A^*}}\uphar\ker D_A
\end{bmatrix};\sD_A,\sD_{A^*}, \ker D_A\right\}
\]
and
\[
\zeta_2=\left\{\begin{bmatrix}P_{\sD_{A^*}}\uphar\sD_A&P_{\sD_{A^*}}A^*\uphar\ker
D_{A^*}\cr
 P_{\ker D_{A^*}}\uphar\sD_{A}
& P_{\ker D_{A^*}}A^*\uphar\ker D_{A^*}
\end{bmatrix};\sD_A,\sD_{A^*}, \ker D_{A^*}\right\}.
\]
By Theorem \ref{Schuriso} the systems $\zeta_1$ and $\zeta_2$ are
conservative and simple and their transfer functions are exactly the
first Schur iterate $\Phi_1(\lambda)$ of $\Phi(\lambda)$.
Note (see \eqref{hnm} and \eqref{Anm}) that
\[
A^*P_{\ker D_{A^*}}\uphar\ker D_A=A^*_{1,0},\;P_{\ker D_{A^*}}A^*\uphar\ker D_{A^*}=A^*_{0,1}.
\]
 Applying Proposition \ref{PURE} we get that the pure part of $\Phi_1(\lambda)$ coincides
with the characteristic functions of the operators $A_{1,0}$ and $ A_{0,1}$.

By Theorem \ref{RELATT} completely non-unitary contractions
$\{A_{n-k,k}\}_{k=0}^n$ are unitarily equivalent. Assume that their
characteristic functions coincide with the pure part of the $n$-th
Schur iterate $\Phi_n(\lambda)$ of $\Phi$. The first Schur iterate of $\Phi_n$ is the function $\Phi_{n+1}(\lambda)$.
 As is already proved above the pure part of $\Phi_{n+1}$ coincides with the characteristic function of the operators
 $(A_{n-k,k})_{1,0}$ and $(A_{n-k,k})_{0,1}$. From \eqref{UE1} it follows
 \[
 (A_{n-k,k})_{1,0}=A_{n+1-k,k},\;(A_{n-k,k})_{0,1}=A_{n-k,k+1}=A_{n+1-(k+1),
k+1}.
 \]
Thus, characteristic functions of the unitarily equivalent completely
non-unitary contractions $\{A_{n+1-k,k}\}_{k=0}^{n+1}$ coincide with $\Phi_{n+1}$.

Note that the M\"obius parameter of the $n-1$-th Schur iterate
$\Phi_{n-1}$ is $\lambda\Phi_n(\lambda)$ and by Theorem \ref{EXX}
this function coincides with the characteristic function of the
operator $\cA_{n,0}=A_{n,0}P_{\ker D_{A_{n,0}}}$. Applying Theorem
\ref{EXX} once again, we get that $A_{n,0}$ is a unilateral shift if
and only if $\Gamma_n$ is a isometry.

The function $\Phi^*(\bar\lambda)$ is the characteristic function of
the operator $A^*$ and its Schur parameters are adjoint to the
corresponding Schur parameters of $\Phi$. In addition if $B=A^*$
then $B_{n,m}=A^*_{m,n}$. Therefore, $A^*_{0, n}$ is a unilateral
shift if and only if $\Gamma^*_n$ is isometric. But $A^*_{0,n}$ is
unuitarily equivalent to $A^*_{n,0}$. Hence, $A_{n,0}$ is a co-shift
if and only if $\Gamma_n$ is a co-isometry.

It follows that $\Gamma_n$ is a unitary if and only if $A_{n,0}$ is
a unilateral shift and co-shift in $\sH_{n,0}$ $\iff$
$\sH_{n,0}=\{0\}$.

Condition \eqref{SHCOSH} holds true if and only if $\Phi_n$ is
identically equal zero. This is equivalent to the condition that
$A_{n,0}$ (as well and $A_{n-1,1},$ $A_{n-2,2},$ $\ldots A_{0,n})$
is the orthogonal sum of a shift and co-shift.
\end{proof}
\begin{remark}
\label{RRREM} It is proved that
\[
\begin{array}{l}
\Gamma_n\;\mbox{is isometry}\; \iff \ker D_{A^{n+1}}=\ker
D_{A^n}\iff\ker D_{A^n}\cap \ker D_{A^*}=\ker D_{A^{n-1}}\cap \ker
D_{A^{*}}\\
\iff \ldots\iff
 \ker D_{A^{n+1-k}}\cap\ker D_{A^{*k}}=\ker D_{A^{n-k}}\cap\ker D_{A^{*k}}\iff\ldots\\
 \iff  \ker D_{A^{*n}}\subset\ker D_{A};
\end{array}
\]
 \[
\begin{array}{l}
\Gamma^*_n \;\mbox{is isometry}\iff \ker D_{A^*}\subset\ker
D_{A^n}\iff \ker D_{A^{n-1}}\cap \ker D_{A^{*2}}=\ker
D_{A^{n-1}}\cap \ker D_{A^*}\\
\iff \ldots\iff
 \ker D_{A^{n-k}}\cap\ker D_{A^{*{k+1}}}=\ker D_{A^{n-k}}\cap\ker D_{A^{*k}}\\
 \iff \ldots\iff \ker D_{A^{{*n+1}}}=\ker
 D_{A^{*n}};
 \end{array}
 \]
\[
\eqref{SHCOSH}\iff \left\{
\begin{array}{l}
\ker D_{A^n}=\left(\bigcap\limits_{l\ge 1}\ker
D_{A^{l}}\right)\oplus\left(\bigcap\limits_{l\ge 1}\ker D_{A^{*l}}\right),\\
P_{\ker D_{A^n}}A\left(\bigcap\limits_{l\ge 1}\ker
D_{A^{*l}}\right)\subset\left(\bigcap\limits_{l\ge 1}\ker
D_{A^{*l}}\right).
\end{array}
\right.
\]
\end{remark}
\subsection{Conservative realizations of the Schur iterates}
\begin{theorem}
\label{ITERATES11} Let $\Theta(\lambda)\in {\bf S}(\sM,\sN)$ and let
\[
\tau_0=\left\{\begin{bmatrix}\Gamma_0&C\cr
B&A\end{bmatrix};\sM,\sN,\sH\right\}
\]
be a simple conservative realization of $\Theta$. Then the Schur
parameters $\{\Gamma_n\}_{n \ge 1}$ of $\Theta$ can be calculated as
follows
\begin{equation}
\label{gamman}
\begin{array}{l}
\Gamma_1=D^{-1}_{\Gamma^*_0}C\left(D^{-1}_{\Gamma_0}B^*\right)^*,�\;
\Gamma_2=D^{-1}_{\Gamma^*_1}D^{-1}_{\Gamma^*_0}CA
\left(D^{-1}_{\Gamma_1}D^{-1}_{\Gamma_0}\left(B^*\uphar\sH_{1,0}\right)\right)^*,\ldots,\\
\Gamma_n=D^{-1}_{\Gamma^*_{n-1}}\cdots
D^{-1}_{\Gamma^*_0}CA^{n-1}\left(D^{-1}_{\Gamma_{n-1}}\cdots
D^{-1}_{\Gamma_0}\left(B^*\uphar\sH_{n-1,0}\right)\right)^*,\ldots.
\end{array}
\end{equation}
Here the operator
\[
\left(D^{-1}_{\Gamma_{n-1}}\cdots
D^{-1}_{\Gamma_0}\left(B^*\uphar\sH_{n-1,0}\right)\right)^*\in\bL(\sD_{\Gamma_{n-1}},\sH_{n-1,0})
\]
is  the adjoint to the operator
\[
D^{-1}_{\Gamma_{n-1}}\cdots
D^{-1}_{\Gamma_0}\left(B^*\uphar\sH_{n-1,0}\right)\in\bL(\sH_{n-1,0},\sD_{\Gamma_{n-1}}),
\]
and
\[\begin{array}{l}
\ran\left(D^{-1}_{\Gamma_{n-1}}\cdots
D^{-1}_{\Gamma_0}\left(B^*\uphar\sH_{n,0}\right) \right)\subset\ran
D_{\Gamma_n},\\
\ran\left(D^{-1}_{\Gamma^*_{n-1}}\cdots
D^{-1}_{\Gamma^*_0}\left(C\uphar\sH_{0,n}\right)\right)\subset\ran
D_{\Gamma^*_n}
\end{array}
\]
for every $n\ge 1$.
 Moreover, for each $n\ge 1$ the unitarily equivalent simple conservative systems
\begin{equation}
\label{taun}
\begin{array}{l}
\tau^{(k)}_{n}=\left\{\begin{bmatrix}\Gamma_n&D^{-1}_{\Gamma^*_{n-1}}
\cdots D^{-1}_{\Gamma^*_{0}}(CA^{n-k})\cr
A^k\left(D^{-1}_{\Gamma_{n-1}}\cdots D^{-1}_{\Gamma_{0}}
\left(B^*\uphar\sH_{n,0}\right)\right)^*&A_{n-k,k}\end{bmatrix};
\sD_{\Gamma_{n-1}},\sD_{\Gamma^*_{n-1}}, \sH_{n-k,k}\right\},\\
k=0,1,\ldots,n \end{array}
\end{equation}
are realizations of the $n$-th Schur iterate $\Theta_n$ of $\Theta$.
 Here the operator
\[
B_n=\left(D^{-1}_{\Gamma_{n-1}}\cdots D^{-1}_{\Gamma_{0}}
\left(B^*\uphar\sH_{n,0}\right)\right)^*\in\bL(\sD_{\Gamma_{n-1}},\sH_{n,0})
\]
is the adjoint to the operator
\[
D^{-1}_{\Gamma_{n-1}}\cdots D^{-1}_{\Gamma_{0}}
\left(B^*\uphar\sH_{n,0}\right)\in\bL(\sH_{n,0},\sD_{\Gamma_{n-1}}).
\]
\end{theorem}
\begin{proof} We will prove by induction.
For $n=1$ it is already established (see Remark \ref{RREM1},
\eqref{RealFirst}, and \eqref{n1}) that
\[
\Gamma_1=D^{-1}_{\Gamma^*_0}C\left(D^{-1}_{\Gamma_0}B^*\right)^*
\]
and the systems
\[
\tau^{(0)}_1=\left\{\begin{bmatrix}\Gamma_1&
D^{-1}_{\Gamma^*_{0}}(CA)\cr \left( D^{-1}_{\Gamma_{0}}
\left(B^*\uphar\sH_{1,0}\right)\right)^*&A_{1,0}\end{bmatrix};
\sD_{\Gamma_0},\sD_{\Gamma^*_{0}}, \sH_{1,0}\right\}
\]
and
\[
\tau^{(1)}_1=\left\{\begin{bmatrix}\Gamma_1&
D^{-1}_{\Gamma^*_{0}}(C)\cr A\left( D^{-1}_{\Gamma_{0}}
\left(B^*\uphar\sH_{1,0}\right)\right)^*&A_{0,1}\end{bmatrix};
\sD_{\Gamma_0},\sD_{\Gamma^*_{0}}, \sH_{0,1}\right\}
\]
are conservative and simple realizations of $\Theta_1.$ Suppose
\[
\tau^{(0)}_m=\left\{\begin{bmatrix}\Gamma_m&D^{-1}_{\Gamma^*_{m-1}}
\cdots D^{-1}_{\Gamma^*_{0}}(CA^m)\cr
\left(D^{-1}_{\Gamma_{m-1}}\cdots D^{-1}_{\Gamma_{0}}
\left(B^*\uphar\sH_{m,0}\right)\right)^*&A_{m,0}\end{bmatrix};
\sD_{\Gamma_{m-1}},\sD_{\Gamma^*_{m-1}}, \sH_{m,0}\right\}
\]
is a simple conservative realization of $\Theta_m$. Then
\[
\begin{array}{l}
B_m=\left(D^{-1}_{\Gamma_{m-1}}\cdots D^{-1}_{\Gamma_{0}}
\left(B^*\uphar\sH_{m,0}\right)\right)^*\in\bL(\sD_{\Gamma_{m-1}},\sH_{m,0}),\\
C_m=D^{-1}_{\Gamma^*_{m-1}} \cdots
D^{-1}_{\Gamma^*_{0}}(CA^m)\in\bL(\sH_{m,0},\sD_{\Gamma^*_{m-1}}),\;
A_{m,0}\in\bL(\sH_{m,0},\sH_{m,0}).
\end{array}
\]
Hence
\[
B^*_m=D^{-1}_{\Gamma_{m-1}}\cdots D^{-1}_{\Gamma_{0}}
\left(B^*\uphar\sH_{m,0}\right)\in\bL(\sH_{m,0},\sD_{\Gamma_{m-1}}).
\]
The first Schur iterate of $\Theta_m(\lambda)$ is the function
$\Theta_{m+1}(\lambda)\in{\bf S}(\sD_{\Gamma_m},\sD_{\Gamma^*_m})$
and the first Schur parameter of $\Theta_m$ is $\Gamma_{m+1}.$ From
\eqref{AKNM} and \eqref{UE1} it follows that
\[
\ker D_{A_{m,0}}=\sH_{m+1,0},\;
(A_{m,0})_{1,0}=A_{m+1,0}\in\bL(\sH_{m+1,0},\sH_{m+1,0}).
\]
Hence by \eqref{RealFirst}, and \eqref{n1}
\[
\Gamma_{m+1}=D^{-1}_{\Gamma^*_{m}}C_m\left(D^{-1}_{\Gamma_{m}}B^*_m\right)^*=
D^{-1}_{\Gamma^*_{m}}\cdots
D^{-1}_{\Gamma^*_0}CA^m\left(D^{-1}_{\Gamma_{m}}\cdots
D^{-1}_{\Gamma_0}\left(B^*\uphar\sH_{m,0}\right)\right)^*
\]
and the system
\[
\tau^{(0)}_{m+1}=\left\{\begin{bmatrix}\Gamma_{m+1}&D^{-1}_{\Gamma^*_{m}}
\cdots D^{-1}_{\Gamma^*_{0}}(CA^{m+1})\cr
\left(D^{-1}_{\Gamma_{m}}\cdots D^{-1}_{\Gamma_{0}}
\left(B^*\uphar\sH_{m+1,0}\right)\right)^*&A_{m+1,0}\end{bmatrix};
\sD_{\Gamma_{m}},\sD_{\Gamma^*_{m}}, \sH_{m+1,0}\right\}
\]
is a simple conservative realization of $\Theta_{m+1}$. From
Proposition \ref{RELAT} it follows that the system
\[
\tau^{(k)}_{m+1}=\left\{\begin{bmatrix}\Gamma_m&D^{-1}_{\Gamma^*_{m}}
\cdots D^{-1}_{\Gamma^*_{0}}(CA^{m+1-k})\cr
A^k\left(D^{-1}_{\Gamma_{m}}\cdots D^{-1}_{\Gamma_{0}}
\left(B^*\uphar\sH_{m+1,0}\right)\right)^*&A_{m+1-k,k}\end{bmatrix};
\sD_{\Gamma_{m}},\sD_{\Gamma^*_{m}}, \sH_{m+1-k,k}\right\}
\]
is unitarily equivalent to the system $\tau^{(0)}_{m+1}$ for
$k=1,\ldots, m+1$ and hence have transfer functions equal to
$\Theta_{m+1}$. This completes the proof.
\end{proof}

Let us make a few remarks which follow from \eqref{exptran},
Proposition \ref{PURE}, and Theorem \ref{ITERATES}.

 If $D_{\Gamma_N}=0$ and
$D_{\Gamma^*_N}\ne 0$ then $\sD_{\Gamma_n}=0$,
$\Gamma^*_n=0\in\bL(\sD_{\Gamma^*_N},\{0\})$,
$\sD_{\Gamma^*_n}=\sD_{\Gamma^*_N}$, and $\sH_{0,n}=\sH_{0,N}$ for
$n\ge N$. The unitarily equivalent observable conservative systems
\[
\tau^{(k)}_N=\left\{\begin{bmatrix}\Gamma_N&D^{-1}_{\Gamma^*_{N-1}}
\cdots D^{-1}_{\Gamma^*_{0}}(CA^{N-k})\cr 0&A_{N-k,k}\end{bmatrix};
\sD_{\Gamma_{N-1}},\sD_{\Gamma^*_{N-1}}, \sH_{N-k,k}\right\},\;
k=0,1,\ldots,N
\]
have transfer functions $\Theta_N(\lambda)=\Gamma_N$ and the
operators $A_{N-k,k}$ are unitarily equivalent co-shifts of
multiplicity $\dim\sD_{\Gamma^*_N},$ the Schur iterates $\Theta_n$
are null operators from $\bL(\{0\},\sD_{\Gamma^*_N})$ for $n\ge N+1$
and are transfer functions of the conservative observable system
\[
\tau_{N+1}=\left\{\begin{bmatrix}0&D^{-1}_{\Gamma^*_{N-1}} \cdots
D^{-1}_{\Gamma^*_{0}}C\cr 0&A_{0,N}\end{bmatrix};
\{0\},\sD_{\Gamma^*_{N}}, \sH_{0,N}\right\}.
\]
 If $D_{\Gamma^*_N}=0$ and $D_{\Gamma_N}\ne 0$ then
$\sD_{\Gamma^*_n}=0$, $\sD_{\Gamma_n}=\sD_{\Gamma_N}$, and
$\Gamma_n=0\in\bL(\sD_{\Gamma_N},\{0\})$, $\sH_{n,0}=\sH_{N,0}$ for
$n\ge N$. The unitarily equivalent controllable conservative systems
\[
\tau^{(k)}_N=\left\{\begin{bmatrix}\Gamma_N&0\cr
A^k\left(D^{-1}_{\Gamma_{N-1}}\cdots D^{-1}_{\Gamma_{0}}
\left(B^*\uphar\sH_{N,0}\right)\right)^*&A_{N-k,k}\end{bmatrix};
\sD_{\Gamma_{N-1}},\sD_{\Gamma^*_{N-1}}, \sH_{N-k,k}\right\}
\]
have transfer functions $\Theta_N(\lambda)=\Gamma_N$ and the
operators $A_{N-k,k}$ are unitarily equivalent unilateral shifts of
multiplicity $\dim\sD_{\Gamma_N},$ the Schur iterates $\Theta_n$ are
null operators from $\bL(\sD_{\Gamma_N},\{0\})$ for $n\ge N+1$ and
are transfer functions of the conservative controllable system
\[
\tau_{N+1}=\left\{\begin{bmatrix}0&0\cr
\left(D^{-1}_{\Gamma_{N}}\cdots D^{-1}_{\Gamma_{0}}
\left(B^*\uphar\sH_{N+1,0}\right)\right)^* &A_{N,0}\end{bmatrix};
\sD_{\Gamma_{N}},\{0\}, \sH_{N,0}\right\}.
\]

\end{document}